\documentclass[psamsfonts]{amsart}

\usepackage[a4paper,bindingoffset=0.0in,left=.975in,right=.975in,top=1.50in,bottom=1.50in,footskip=.35in]{geometry}

\usepackage{amssymb,amsfonts}
\usepackage[all,arc]{xy}
\usepackage{enumerate}
\usepackage{mathrsfs}
\usepackage{setspace}

\newtheorem{thm}{Theorem}[section]
\newtheorem{cor}[thm]{Corollary}
\newtheorem{prop}[thm]{Proposition}
\newtheorem{lem}[thm]{Lemma}

\newtheorem{quest}[thm]{Question}

\theoremstyle{definition}
\newtheorem{defn}[thm]{Definition}

\newtheorem{notn}[thm]{Notation}

\theoremstyle{remark}
\newtheorem{rem}[thm]{Remark}

\newcommand{\R}{\Bbb{R}}

\newcommand{\Q}{\Bbb{Q}}

\newcommand{\OO}{\mathcal{O}}

\newcommand{\CC}{\mathcal{C}}

\newcommand{\WW}{\mathcal{W}}

\newcommand{\F}{\Bbb{F}}

\newcommand{\Z}{\Bbb{Z}}

\newcommand{\N}{\Bbb{N}}

\newcommand{\m}{\frak{m}}

\newcommand{\res}{\text{res}}

\newcommand{\Mod}{\text{Mod}}

\newcommand{\Ind}{\text{Ind}}

\newcommand{\Fil}{\text{Fil}}

\newcommand{\ind}{\text{ind}}

\newcommand{\gr}{\text{gr}}

\newcommand{\Hom}{\text{Hom}}

\newcommand{\Ext}{\text{Ext}}

\newcommand{\Gal}{\text{Gal}}

\newcommand{\GL}{\text{GL}}

\newcommand{\Rep}{\text{Rep}}

\newcommand{\Tor}{\text{Tor}}

\newcommand{\SL}{\text{SL}}

\newcommand{\y}{\hspace{6pt}}

\makeatletter
\let\c@equation\c@thm
\makeatother
\numberwithin{equation}{section}

\title{A vanishing result for higher smooth duals}

\author{Claus Sorensen}

\date{}

\begin{document}

\begin{abstract}
In this paper we prove a general vanishing result for Kohlhaase's higher smooth duality functors $S^i$. If $G$ is any unramified connected reductive $p$-adic group, $K$ is a hyperspecial subgroup, and $V$ is a Serre weight, we show that $S^i(\ind_K^G V)=0$ for $i>\dim(G/B)$ where $B$ is a Borel subgroup. (Here and throughout the paper $\dim$ refers to the dimension over $\Q_p$.) This is due to Kohlhaase for 
$\GL_2(\Q_p)$ in which case it has applications to the calculation of $S^i$ for supersingular representations. Our proof avoids explicit matrix computations by making use of Lazard theory, and we deduce our result from an analogous statement for graded algebras via a spectral sequence argument. The graded case essentially follows from Koszul duality between symmetric and exterior algebras. 
\end{abstract}

\maketitle

\tableofcontents


\onehalfspacing

\section{Introduction}

Dual representations are ubiquitous in the classical Langlands program. For one thing they appear in the functional equation for automorphic $L$-functions. In the $p$-adic Langlands program, and its counterpart modulo $p$, dual representations play a prominent role as well. Say $\pi$ is a smooth representation of some $p$-adic reductive group $G$, having coefficients in a subfield $E \subset \bar{\F}_p$. Then the full linear dual $\pi^{\vee}=\Hom_E(\pi,E)$ lives in a completely different category of (pseudocompact) modules over the Iwasawa algebra $\Lambda(G)=E[G]\otimes_{E[K]} E[\![K]\!]$ where $K \subset G$ is an arbitrary compact open subgroup. One can of course take the smooth vectors of $\pi^{\vee}$ and get a smooth representation $S^0(\pi)=(\pi^{\vee})^{\infty}$ -- but this is often zero! In fact $S^0(\pi)=0$ for all irreducible admissible $\pi$ which are infinite-dimensional. Moreover, the smooth dual functor $S^0$ is not exact. In \cite{Koh17} Kohlhaase extended $S^0$ to a $\delta$-functor consisting of certain higher smooth duality functors 
$$
S^i: \Rep_E^{\infty}(G) \longrightarrow \Rep_E^{\infty}(G)
$$
defined as $\Sigma^i \circ (\cdot)^{\vee}$ where $\Sigma^i$ extends the smooth vectors functor $(\cdot)^{\infty}$ ("stable cohomology"), and he establishes foundational results on the functorial properties of the $S^i$. Some aspects of this are analogous to (and in fact rely on) the theory of Schneider and Teitelbaum for locally analytic representations \cite{ST05} in which they (among other things) produce an involutive functor from $D_{\text{adm}}^b\big(\Rep_{\mathcal{E}}^{\infty}(G)\big)$ to itself, which is compatible with a natural duality on the derived category of coadmissible modules over the locally analytic distribution algebra $D(G,\mathcal{E})$. Here $\mathcal{E}/\Q_p$ is a $p$-adic field and 
$D_{\text{adm}}^b$ is the bounded derived category of complexes with admissible cohomology, cf. \cite[Cor.~3.7]{ST05} and the diagram at the bottom of p. 317 there. Kohlhaase does not mention it explicitly, but one can easily prove the analogous result for mod $p$ coefficients $E$ and define a functor $S$ from $D^b\big(\Rep_{E}^{\infty}(G)\big)$ to itself which becomes involutive on the subcategory
$D_{\text{adm}}^b\big(\Rep_{E}^{\infty}(G)\big)$. Moreover, there is a spectral sequence starting from $S^i\big(h^j(V^{\bullet})\big)$ converging to
the cohomology of $S(V^{\bullet})$. In fact Emerton gave a series of lectures at the Institut Henri Poincar\'{e} in March 2010 in which he introduced a derived duality functor 
$\Bbb{D}_G$ similar to $S$ which goes into the proof of the Poincar\'{e} duality spectral sequence for completed (Borel-Moore) homology. His derived category "$D^b(G)$" is a bit subtler however.

Schneider proved in \cite[Thm.~9]{Sch15} that the unbounded derived category $D\big(\Rep_{E}^{\infty}(G)\big)$ is equivalent to the derived 
category of differential graded modules over a certain DGA variant of the Hecke algebra $\mathcal{H}_I^{\bullet}$, defined relative to a torsionfree pro-$p$ group $I\subset G$
(and a choice of injective resolution).
A key ingredient was \cite[Prop.~6]{Sch15} which shows that $\ind_I^G 1$ is a (compact) generator of $D\big(\Rep_{E}^{\infty}(G)\big)$. From this point of view, to understand $S$ on the whole derived category of $G$ we should first understand the higher smooth duals of the generator $S^i(\ind_I^G 1)$. In fact this ties to a question posed by Harris in \cite[Q.~4.5]{Har16}
as to whether there is any relation between Kohlhaase's $S^i$ and $E$-linear duality on $D(\mathcal{H}_I^{\bullet})$. We know $S^i=0$ for $i>\dim(G)$. Here and throughout the paper $\dim(\cdot)$ refers to the dimension as a $\Q_p$-manifold. Some of our main motivation for 
writing this paper was to show that at least for the generator $\ind_I^G 1$ one can improve this bound significantly and show that $S^i\big(\ind_I^G 1\big)=0$ for 
$i>\dim(G/B)$ where $B$ is a Borel subgroup. This follows from our main result (Theorem \ref{main}) below by taking $V=\ind_I^K 1$ there. In fact an easy inductive argument shows more generally that 
$$
S^i(\pi)=0 \y \y \y \forall i>\dim(G/B)+\ell
$$
if there is an exact sequence of length $\ell$ of the form
$$
0 \longrightarrow (\ind_I^G 1)^{\oplus r_{\ell}} \longrightarrow \cdots \longrightarrow (\ind_I^G 1)^{\oplus r_1} \longrightarrow (\ind_I^G 1)^{\oplus r_0} \longrightarrow \pi \longrightarrow 0
$$
for some $r_i \in \Z_{>0}$. Of course this bound on $i$ is trivial unless $\ell\leq \dim(B)$.
We do not know in which generality such resolutions exist\footnote{Ollivier has constructed analogous resolutions when $\pi$ is a principal series representation of $\GL_n$ in \cite[Thm.~1.1]{Oll14}.} for a general group $G$. For groups of semisimple rank one, Kohlhaase has constructed a class of representations in \cite{Koh18} to which our bound applies.
We also do not know whether the bound $\dim(G/B)$ is sharp for $\pi=\ind_I^G 1$ in the sense that $S^{\dim(G/B)}(\ind_I^G 1)$ is nonzero -- even in the case of $\GL_2(\Q_p)$.

We now state the main result of this paper which we alluded to above. Let $F/\Q_p$ be a finite extension, and $G_{/F}$ a connected reductive group with a Borel subgroup $B$. We assume that $G$ is unramified\footnote{Experts have informed us this restriction is likely to be unnecessary (with $K$ special and $B$ a minimal parabolic in general).} and choose a hyperspecial maximal compact subgroup $K \subset G$ along with a finite-dimensional smooth representation $K \rightarrow \GL(V)$ with coefficients in some algebraic extension $E/\F_p$. Let $\ind_K^GV$ be the compact induction. Then we have the following vanishing result for its higher smooth duals $S^i$.   

\begin{thm}\label{main}
$S^i(\ind_K^GV)=0$ for all $i>\dim(G/B)$.
\end{thm}

What we actually prove is a slightly stronger result on the vanishing of the transition maps. Namely, if $N \vartriangleleft K$ has an Iwahori factorization and acts trivially on $V$, then the restriction map
$$
\Ext_{\Lambda(N)}^i\big(E, (\ind_K^GV)^{\vee}\big) \longrightarrow \Ext_{\Lambda(N^{p^m})}^i\big(E, (\ind_K^GV)^{\vee}\big)
$$ 
vanishes for $i>\dim(G/B)$ as long as $m$ is greater than some constant depending only on $i$ and $V,N$ (and an auxiliary filtration on $N$). Here $(-)^{\vee}$ denotes Pontryagin duality, and $\Lambda(N)=E[\![N]\!]$ is the completed group algebra over $E$. Note that the {\it{set}} of $p^m$-powers $N^{p^m}$ is a group for $m$ large enough for any $p$-valuable group $N$, cf. \cite[Rem.~26.9]{Sch11}.

We have been unable to show that $S^{\dim(G/B)}(\ind_K^GV)\neq 0$ but we believe this should be true under a suitable regularity condition on the weight $V$, cf. section \ref{nv}. We hope to address this in future work, and to say more about the action of Hecke operators on $S^i(\ind_K^GV)$ for all $i$.
Let us add that the bound $\dim(G/B)$ is not sharp for all $V$. Indeed $S^i(\ind_K^G1)=0$ for all $i>0$, cf. Remark \ref{trivial}.

For $\GL_2(\Q_p)$ Theorem \ref{main} amounts to \cite[Thm.~5.11]{Koh17}, which is one of the main results of that paper. There is a small difference coming from the center 
$Z \simeq \Q_p^{\times}$ though. He assumes $V$ is an irreducible representation of $K=\GL_2(\Z_p)$ which factors through $\GL_2(\F_p)$, extends the central character of $V$ to $Z$ by sending $p \mapsto 1$, and considers $\ind_{KZ}^G V$ instead. The latter carries a natural Hecke operator $T=T_V$ whose cokernel $\pi_V$ is an irreducible supersingular representation. As $V$ varies this gives all supersingular representations of $\GL_2(\Q_p)$ (with $p$ acting trivially), cf. \cite[Prop.~4]{BL94} and \cite[Thm.~1.1]{Bre03}. The short exact sequence
$$
0 \longrightarrow \ind_{KZ}^G V \overset{T}{\longrightarrow} \ind_{KZ}^G V \longrightarrow \pi_V \longrightarrow 0
$$
gives rise to a long exact sequence of higher smooth duals
$$
\cdots \longrightarrow S^i(\pi_V) \longrightarrow S^i(\ind_{KZ}^G V) \overset{S^i(T)}{\longrightarrow} S^i(\ind_{KZ}^G V) \longrightarrow S^{i+1}(\pi_V) \longrightarrow \cdots.
$$
Since $S^i(\ind_{KZ}^G V)=0$ for $i>1$ it follows immediately that $S^3(\pi_V)=S^4(\pi_V)=\cdots=0$. Kohlhaase takes this a step further and shows that also $S^2(\pi_V)=0$ by checking that $S^1(T)$ is bijective (\cite[p.~45]{Koh17}), which is also a key ingredient in his long verification that $S^1(\pi_V)\simeq \pi_{\check{V}}$. The latter is another of his main results, \cite[Thm.~5.13]{Koh17}. Essentially what this shows is that via the mod $p$ local Langlands correspondence $\rho \rightsquigarrow \pi(\rho)$ the dual Galois representation corresponds to $S^1$. More precisely $\pi(\rho^{\vee})\otimes \omega_{\text{cyc}}=S^1\big(\pi(\rho)\big)$ for Galois representations $\rho: \Gal_{\Q_p}\rightarrow \GL_2(E)$ (not a twist of an extension of $\omega_{\text{cyc}}$ by $1$). See the conclusion in \cite[Rem.~5.15]{Koh17}.

Similar descriptions of the supersingular representations exist for other rank one groups. In \cite{Abd14} Abdellatif classifies the supersingulars of 
$\SL_2(\Q_p)$ by decomposing $\pi_V|_{\SL_2(\Q_p)}$ into irreducibles, and in \cite{Koz16} Koziol was able to bootstrap from this case and deal with 
the unramified unitary group $U(1,1)$ over $\Q_p$. As an application of our Theorem \ref{main} there is forthcoming work of Jake Postema for his UC San Diego Ph.D. thesis in which he calculates all the $S^i$ for supersingular representations of both 
$\SL_2(\Q_p)$ and $U(1,1)$, and explores how $S^1$ flips the members of an $L$-packet. 

Let us give a roadmap of our proof of Theorem \ref{main} and then point out some analogies and discrepancies with Kohlhaases's approach. First of all $(\ind_K^GV)^{\vee}$ is the full induction $I_K^G \check{V}$ and we need to understand the restriction maps between the various $\Ext_{\Lambda(N)}^i\big(E, I_K^G \check{V} \big)$ for $N$ deep enough. We start with a normal subgroup $N \vartriangleleft K$ which has an Iwahori factorization (relative to $B$). This means among other things that multiplication defines a homeomorphism
\begin{equation}\label{if}
(N \cap \bar{U}) \times (N \cap T) \times (N \cap U) \overset{\sim}{\longrightarrow} N,
\end{equation}
where $B=TU$. Moreover, conjugation by $s \in S^+$ contracts and expands the rightmost and leftmost factors respectively. Here $S\subset T$ is a maximal split subtorus, and $S^+$ is the monoid of elements $s\in S$ satisfying $|\alpha(s)|\leq 1$ for all roots $\alpha>0$. Essentially by the Cartan decomposition, we get an isomorphism of
$E[K]$-modules
$$
\Ext^i_{\Lambda(N)}(E,I_K^G \check{V}) \overset{\sim}{\longrightarrow} {\prod}_{s\in S^+/S(\OO)} \ind_{NK^s}^K \Ext^i_{\Lambda(N \cap K^s)}(E, \check{V}^s),
$$
compatible with restriction maps on either side. Here $K^s:=K \cap s^{-1}Ks$, and $\check{V}^s$ denotes the representation of $K^s$ on $\check{V}$ obtained by 
$\kappa \ast_s \check{v}:=(s\kappa s^{-1})\check{v}$. It therefore suffices to understand the restriction maps between the various $\Ext^i_{\Lambda(N \cap K^s)}(E, \check{V}^s)$
well enough, and control what happens when we vary $s$. For the sake of exposition let us first outline the argument in the case $s=1$. Here the goal is to show the vanishing of 
$$
\Ext^i_{\Lambda(N)}(E, \check{V}) \longrightarrow \Ext^i_{\Lambda(N')}(E, \check{V})
$$
for $i>\dim(G/B)$ and $N'\subset N$ sufficiently small. To run the argument below for general $s$ it is important {\it{not}} to assume $N$ acts trivially on $V$ here (otherwise the map would vanish for $i>0$ and $N'=N^p$). The point is that $s$ contracts $N \cap U$ but expands $N \cap \bar{U}$. Thus we can safely assume $N \cap U$ acts trivially on $V$, but we do not know this for the factor $N \cap \bar{U}$ (after conjugating by $s$). 

The basic idea is to prove the analogous vanishing result for certain graded algebras $\gr \Lambda(N)$ and then employ a spectral sequence argument elaborated upon in Appendix A. This is based on Lazard theory. There is a natural way to endow $N$ with a $p$-valuation $\omega$ which behaves well with respect to the Iwahori factorization, and although (\ref{if}) is not an isomorphism of groups it does become one once we pass to graded groups. Consequently, there is an isomorphism
of graded algebras
$$
\gr \Lambda(N \cap \bar{U})\otimes \gr \Lambda(N \cap T) \otimes \gr \Lambda(N \cap U) \overset{\sim}{\longrightarrow} \gr \Lambda(N),
$$
and one obtains a K\"{u}nneth formula for $\Ext^i_{\gr \Lambda(N)}(E, \gr \check{V})$ expressing it as a direct sum
$$
{\bigoplus}_{i_1+i_2+i_3=i} \Ext^{i_1}_{\gr \Lambda(N\cap \bar{U})}(E, \gr \check{V})\otimes \Ext^{i_2}_{\gr \Lambda(N \cap T)}(E,E)\otimes  \Ext^{i_3}_{\gr \Lambda(N \cap U)}(E,E).
$$
The $(i_2,i_3)$-factors can be dealt with using the Koszul resolution, which shows for instance that
$$
 \Ext^{i_2}_{\gr \Lambda(N \cap T)}(E,E)\simeq \bigwedge^{i_2} \big(E \otimes_{\F_p[\pi]} \gr \Lambda(N \cap T)\big)^*
 $$
 where $\pi$ is the usual $p$-power operator from Lazard theory. We have tacitly perturbed $\omega$ if necessary to make the Lie algebra $\gr N$ abelian. The restriction map $\Ext^i_{\gr \Lambda(N)}(E, \gr \check{V})\rightarrow \Ext^i_{\gr \Lambda(N^p)}(E, \gr \check{V})$ therefore vanishes on the $(i_1,i_2,i_3)$-summand unless $i_2=i_3=0$.
So we only need to worry about the summand with $i_1=i$. But the factor $\Ext^{i}_{\gr \Lambda(N\cap \bar{U})}(E, \gr \check{V})$ itself vanishes in the range $i>\dim(\bar{U})$ by well-known results on the cohomological dimension. This is where the bound $\dim(G/B)=\dim(\bar{U})$ comes from. Altogether this shows that
$$
\Ext^i_{\gr \Lambda(N)}(E, \gr \check{V})\overset{0}{\longrightarrow} \Ext^i_{\gr \Lambda(N^p)}(E, \gr \check{V}) \y \y \y \forall i>\dim(G/B).
$$
Since $\Lambda(N)$ is a filtered algebra, its bar resolution inherits a filtration which results in a convergent spectral sequence (cf. Appendix A where this is explained in detail)
which is functorial in $N$,
$$
E_1^{i,j}=\Ext^{i,j}_{\gr \Lambda(N)}(E, \gr \check{V}) \Longrightarrow \Ext^{i+j}_{\Lambda(N)}(E, \check{V}).
$$
One deduces that the restriction map $\Ext^i_{\Lambda(N)}(E, \check{V}) \rightarrow \Ext^i_{\Lambda(N^p)}(E, \check{V})$ is zero on the graded pieces, that is it takes $\Fil^j$ to 
$\Fil^{j+1}$ for all $j$. This filtration is both exhaustive and separated, so doing this repeatedly shows that $\Ext^i_{\Lambda(N)}(E, \check{V}) \rightarrow \Ext^i_{\Lambda(N^{p^m})}(E, \check{V})$ vanishes for large enough $m$. In fact we can control how large (and this is crucial when we vary $s$). We may assume $\gr \Lambda(N)$ is Koszul (even a polynomial algebra) in which case we understand the internal grading $\Ext^{i,j}_{\gr \Lambda(N)}$ better and it suffices to take $m>\mu+i$ where $\mu$ is the first index for which 
$\Fil^{\mu}\Lambda(N)\check{V}=0$. We end with a few remarks on the case of arbitrary $s \in S^+$. The argument above still works, and shows that
$$
\Ext^i_{\Lambda(N \cap K^s)}(E, \check{V}^s)\overset{0}{\longrightarrow} \Ext^i_{\Lambda(N^{p^m}\cap K^s)}(E, \check{V}^s) \y \y \y \forall i>\dim(G/B)
$$
as long as $m>\mu +i$. It is of course extremely important that our invariant $\mu$ is independent of $s$. Getting this lower bound on $m$ which is uniform in $s \in S^+$ is indeed a key point of the whole paper.

To guide the reader we now point out some similarities and differences between our paper and \cite{Koh17}. Overall Kohlhaase argues more from the point of view of 
\cite{DdSMS} whereas we lean more towards \cite{Laz65}. His groups $G_n$ are analogous to our $K^s$ (the difference being the center $Z$). On the other hand, where we choose to work with $N^{p^m}$ he is considering the full congruence subgroups $K_m=\ker\big(\GL_2(\Z_p)\rightarrow \GL_2(\Z/p^m\Z)\big)$. The first few formal steps, using the Cartan and Mackey decompositions, are identical. In Kohlhaase the problem is reduced to understanding the restriction maps between $\Ext_{\Lambda(\mathcal{U}_{m,n})}^i(E, \check{V})$
where $\mathcal{U}_{m,n}=\alpha^{-n}(G_n \cap K_m)\alpha^n$ corresponds to our $K \cap sN^{p^m}s^{-1}$. The key point where Kohlhaase uses he is in the $\GL_2(\Q_p)$ situation is to show that these groups $\mathcal{U}_{m,n}$ are uniform pro-$p$ groups (certain finitely generated torsion-free groups $\mathcal{U}$ for which 
$[\mathcal{U},\mathcal{U}]\subset \mathcal{U}^p$ -- at least for $p>2$). Based on the Iwahori factorization of $\mathcal{U}_{m,n}$ this becomes case-by-case explicit matrix computations exhibiting various commutators as $p$-powers (cf. \cite[p.~40]{Koh17}). One of the goals of our paper was to do this is a more conceptual way using Lazard's theory of $p$-valuations. Knowing $\mathcal{U}_{m,n}$ is uniform (for $m \geq 2$) Kohlhaase invokes \cite[Thm.~7.24]{DdSMS} to get an isomorphism
$\gr \Lambda(\mathcal{U}_{m,n})\simeq E[X_1,X_2,X_3,X_4]$, where one puts the $\m$-adic filtration on $\Lambda(\mathcal{U}_{m,n})$. This isomorphism is also a fundamental result in Lazard theory where more generally one identifies $\gr \Lambda(\mathcal{U})$ with $U(\gr \mathcal{U})$ for any $p$-valued group $\mathcal{U}$. At this point our calculations on the "graded side" more or less agree. However, to pass from the graded situation and deduce vanishing of the restriction maps for $\Lambda(\mathcal{U}_{m,n})$ itself Kohlhaase employs a result of Gr\"{u}nenfelder (\cite{Gru79}) to the effect that there are functorial isomorphisms
\begin{equation}\label{gru}
\Ext_{\Lambda(\mathcal{U}_{m,n})}^i(E, \check{V}) \overset{\sim}{\longrightarrow} \Ext_{\gr \Lambda(\mathcal{U}_{m,n})}^i(E, \gr \check{V})
\end{equation}
dual to an analogue for $\Tor_i$ which supposedly is \cite[Thm.~3.3']{Gru79}. We have not been able to understand the details of Gr\"{u}nenfelder's paper, and we have been unsuccessful\footnote{In fact the first paragraph of \cite[Sect.~3.4]{PT02} indicates there are counterexamples to \cite[Thm.~3.3']{Gru79}.}
in establishing the isomorphism (\ref{gru}) directly for non-trivial weights $V$; even for $i=0$. Instead we use a spectral sequence similar to the one used in \cite[p.~72]{Bjo87} of the form 
$$
E_1^{i,j}=\Ext_{\gr \Lambda(\mathcal{U})}^{i,j}(E,\gr \check{V}) \Longrightarrow \Ext_{\Lambda(\mathcal{U})}^{i+j}(E,\check{V}).
$$
However, \cite{Bjo87} does not contain quite what we need. To get a spectral sequence which is functorial in $\mathcal{U}$ we have to work with bar resolutions instead of the resolutions arising from "good" filtrations as in \cite{Bjo87}. This spectral sequence and its applications are discussed at great length in Appendix A of our paper, which also treats the Koszul case where the sequence simplifies due to the vanishing of most $E_1^{i,j}$. 

\section{Notation}\label{not}

Let $p$ be a prime. Fix a finite extension $F/\Q_p$ with ring of integers $\OO$. Choose a uniformizer $\varpi$ and let $k=\OO/\varpi \OO$ be the residue field, with cardinality
$q=p^f$. We normalize the absolute value $|\cdot|$ on $F$ such that $|\varpi|=q^{-1}$. Thus $|p|=q^{-e}$ where $e$ is the ramification index. 

We take $G_{/F}$ to be a connected reductive $F$-group, assumed to be unramified. The latter means $G$ is quasi-split (over $F$) and $G \times_F F'$ is split for some 
finite unramified extension $F'/F$. Such groups admit integral models, cf. \cite[3.8]{Tit79}. More precisely $G$ extends to a smooth affine group scheme $G_{/\OO}$ for which 
$G \times_{\OO} k$ is connected and reductive. We fix such a model once and for all and let $K=G(\OO)$ be the corresponding hyperspecial maximal compact subgroup of $G(F)$. 
When there is no risk of confusion we will just write $G$ instead of $G(F)$; similarly for any linear algebraic $F$-group.

Choose a maximal $F$-split subtorus $S \subset G$ and let $T=Z_G(S)$ which is then a maximal $F$-subtorus. They extend naturally to smooth $\OO$-subgroup schemes $S\subset T$ of $G$ (which reduce to a maximal $k$-split torus and its centralizer in the special fiber). Let $\Phi=\Phi(G,S)\subset X^*(S)$ be the relative root system, which we interchangeably view in the generic or special fiber. Once and for all we choose a system of positive roots $\Phi^+$ and let $B=T \ltimes U$ be the corresponding Borel subgroup, which also extends to a closed $\OO$-subgroup scheme of $G_{/\OO}$. We let $\bar{B}=T \ltimes \bar{U}$ denote the opposite Borel subgroup ($B \cap \bar{B}=T$).  


\section{The dual of compact induction}\label{ind}

Throughout the paper we let $E/\F_p$ be a fixed algebraic extension which will serve as the coefficient field of our representations. We start with a finite-dimensional smooth $E$-representation $V$ of $K=G(\OO)$. Note that $K_r=\ker\big( G(\OO)\rightarrow G(\OO/\varpi^r \OO)\big)$ necessarily acts trivially for $r$ sufficiently large. Let $r_V$ be the smallest 
$r\geq 1$ for which $K \rightarrow \GL(V)$ factors through $K/K_r$. In applications $V$ will typically be a Serre weight, meaning it arises from an absolutely irreducible rational  representation of $G \times_{\OO} E$ (via some choice of embedding $k \hookrightarrow E$) but this restriction is unnecessary. 

Let $\check{V}=\Hom_E(V,E)$ be the full $E$-linear dual with the contragredient $K$-action. This is again a smooth representation, and obviously $r_{\check{V}}=r_V$. 
We let $\langle,\rangle_V$ denote the natural pairing $\check{V}\times V \rightarrow E$. 

Our conventions on induced representations differ from \cite[p.~38]{Koh17}. We let $G$ act by right translations, whereas Kohlhaase lets $G$ act by left translations. 

\begin{defn}
$I_K^GV$ is the full induction of \underline{all} functions $f: G \rightarrow V$ such that $f(\kappa g)=\kappa f(g)$ for all $\kappa \in K$ and $g\in G$.
We let $\Ind_K^GV$ be the subrepresentation of smooth functions $f$ (i.e., invariant under right-translation by an open subgroup). Finally $\ind_K^G V\subset \Ind_K^GV$ 
consists of smooth functions $f$ which are compactly supported (so $\text{supp}(f)$ is the union of finitely many cosets $Kg$). 
\end{defn}

\begin{rem}
Upon choosing a set of coset representatives $(g_i)_{i \in I}$ for $K \backslash G$, the map $f \mapsto (f(g_i))_{i \in I}$ defines isomorphisms of $E$-vector spaces 
$$
I_K^GV \overset{\sim}{\longrightarrow} \prod_{i \in I} V \y \y \y \text{\underline{and}} \y \y \y \ind_K^G V \overset{\sim}{\longrightarrow} \bigoplus_{i \in I} V.
$$
\end{rem}

Pontryagin duality $(\cdot)^{\vee}$ sets up an anti-equivalence between the category of smooth representations $\Rep_E^{\infty}(G)$ and the category $\Mod_{\Lambda(G)}^{\text{pc}}$ 
of pseudocompact $E$-vector spaces $M$ with a jointly continuous action $G \times M \rightarrow M$, cf. \cite[Thm.~1.5]{Koh17}. The $E[G]$-module structure of any such $M$ extends to a $\Lambda(G)$-module structure, where $\Lambda(G):=E[G]\otimes_{E[K]} E[\![K]\!]$ is the Iwasawa algebra. Here $E[\![K]\!]=\varprojlim_N E[K/N]$ is the completed group algebra, with $N \vartriangleleft K$ ranging over the open normal subgroups.

The compact induction $\ind_K^G V$ is a (non-admissible) object of $\Rep_E^{\infty}(G)$. We first identify its Pontryagin dual with $I_K^G \check{V}$. 
This is \cite[Lem.~5.9]{Koh17}; we include the argument for convenience. 

\begin{lem}
There is an isomorphism $I_K^G \check{V} \overset{\sim}{\longrightarrow} (\ind_K^G V)^{\vee}$ of pseudocompact $\Lambda(G)$-modules.
\end{lem}

\begin{proof}
We define a $G$-equivariant pairing $I_K^G \check{V} \times \ind_K^G V \rightarrow E$ by the following (finite) sum
$$
\langle \check{f},f\rangle:={\sum}_{g\in K \backslash G} \langle \check{f}(g),f(g)\rangle_V.
$$
The resulting map $I_K^G \check{V} \rightarrow (\ind_K^G V)^{\vee}$ corresponds to the natural map $\prod_{i\in I}\check{V} \rightarrow (\bigoplus_{i \in I} V)^{\vee}$ after choosing coset representatives $(g_i)_{i \in I}$. The latter is trivially an isomorphism (of $E$-vector spaces).
\end{proof}

Consequently, the compact induction $\ind_K^G V$ has smooth dual the smooth induction of $\check{V}$,
$$
S^0(\ind_K^G V) \simeq \Ind_K^G \check{V}.
$$
The ultimate goal of this paper is to get a better understanding of the higher smooth duals $S^i(\ind_K^G V)$, 
cf. \cite[Def.~3.12]{Koh17}. In particular, for which $i>0$ do they vanish? For $G=\GL_2(\Q_p)$ Kohlhaase has shown that 
$S^i(\ind_K^G V)=0$ for $i>1$, cf. \cite[Thm.~5.11]{Koh17}. We extend his result to arbitrary $G$. 


\section{Stable cohomology and the functors $S^i$}

We fix an open $N \vartriangleleft K$ for now and let it act trivially on $E$, which we view as an object of $\Mod_{\Lambda(N)}^{\text{pc}}$. The module structure on $E$ comes from the augmentation map $\epsilon: \Lambda(N)=E[\![N]\!] \rightarrow E$. For other objects $M$ of $\Mod_{\Lambda(N)}^{\text{pc}}$ we will explore how $\Ext^i_{\Lambda(N)}(E,M)$ varies when we shrink $N$. The subscript in $\Ext^i_{\Lambda(N)}$ signifies we take extensions in the abelian category $\Mod_{\Lambda(N)}^{\text{pc}}$ (not the category of abstract 
left $\Lambda(N)$-modules). Eventually we will be focussing on $M=I_K^G \check{V}$ (which is not coadmissible). 

To make the functorial properties of $\Ext^i_{\Lambda(N)}(E,M)$ more transparent we prefer to fix a concrete projective resolution of $E$ which itself is functorial in $N$. We take the topological bar resolution $B_{\bullet} \Lambda(N)$, cf. \cite[V(1.2.9)]{Laz65}. Recall that
$$
B_n \Lambda(N)=\Lambda(N)^{\hat{\otimes}(n+1)}=\Lambda(N) \hat{\otimes}_E \cdots \hat{\otimes}_E \Lambda(N)\simeq \Lambda(N\times \cdots \times N)
$$
is a projective object of $\Mod_{\Lambda(N)}^{\text{pc}}$ by \cite[Cor.~3.3]{Bru66}, and together they form a resolution
$$
\cdots \overset{d_3}{\longrightarrow} B_2 \Lambda(N)\overset{d_2}{\longrightarrow} B_1\Lambda(N) \overset{d_1}{\longrightarrow}  \Lambda(N) \overset{\epsilon}{\longrightarrow} E \longrightarrow 0
$$
whose differentials $d_n$ are defined by a standard formula we will not state here. Note that $K$ acts on $B_{\bullet} \Lambda(N)$ as follows. For each element $\kappa \in K$ consider the conjugation map $\kappa(\cdot)\kappa^{-1}$ on $N$. It induces an automorphism $\kappa: \Lambda(N) \overset{\sim}{\longrightarrow} \Lambda(N)$ and in turn an automorphism of each $B_n \Lambda(N)$ which we henceforth view as an $E[K]$-module this way. 

\begin{notn}
When we write $R\Hom_{\Lambda(N)}(E,M)$ and $\Ext^i_{\Lambda(N)}(E,M)$ we will always mean those obtained from the bar resolution $B_{\bullet} \Lambda(N)$. I.e.,
$$
R\Hom_{\Lambda(N)}(E,M):=\Hom_{\Lambda(N)}(B_{\bullet} \Lambda(N),M), \y \y \y \y \Ext^i_{\Lambda(N)}(E,M):=h^i\big(\Hom_{\Lambda(N)}(B_{\bullet} \Lambda(N),M)\big).
$$
Note that if $M$ carries an $E[K]$-module structure, then $R\Hom_{\Lambda(N)}(E,M)$ becomes a complex of $E[K/N]$-modules
$$
 0 \longrightarrow \Hom_{\Lambda(N)}(\Lambda(N),M) \longrightarrow \Hom_{\Lambda(N)}(B_1\Lambda(N),M) \longrightarrow \Hom_{\Lambda(N)}(B_2\Lambda(N),M) \longrightarrow 
 \cdots
$$
(and each $\Ext^i_{\Lambda(N)}(E,M)$ is naturally an $E[K/N]$-module). Perhaps it is worthwhile to really flesh out the $E[K/N]$-module structure of 
$\Hom_{\Lambda(N)}(B_{n} \Lambda(N),M)$ here. Since $N \vartriangleleft K$ there is a $K$-action on $\Lambda(N)$ an therefore also on $B_{n} \Lambda(N)$ induces by conjugation 
$n \mapsto \kappa n\kappa^{-1}$ on $N$. We let $\kappa \in K$ act on an element $\phi \in \Hom_{\Lambda(N)}(B_{n} \Lambda(N),M)$ in the usual way $(\kappa\phi)(x)=\kappa\phi(\kappa^{-1}x)$. It is trivial to check that $\kappa\phi$ is again $\Lambda(N)$-linear, and clearly $\kappa\phi=\phi$ for $\kappa \in N$ -- thus the $K$-action factors through $K/N$.
\end{notn}

We will now discuss how these vary when we shrink $N$. So let $N'\subset N$ be an open subgroup. The induced map $\Lambda(N') \hookrightarrow \Lambda(N)$ is a homeomorphism onto its image and $\Lambda(N)$ becomes a free module over $\Lambda(N')$ of rank $[N:N']$, see \cite[Cor.~19.4]{Sch11} for instance. In particular, viewing 
$B_{\bullet} \Lambda(N)$ as a complex of $\Lambda(N')$-modules still gives a projective resolution of $E$. 

\begin{lem}\label{he}
The natural map $B_{\bullet} \Lambda(N')\longrightarrow B_{\bullet} \Lambda(N)$ is a homotopy equivalence.
\end{lem}

\begin{proof}
Basic homological algebra. As observed, both $B_{\bullet} \Lambda(N')$ and $B_{\bullet} \Lambda(N)$ are projective resolutions of $E$ (viewed as a $\Lambda(N')$-module).
Therefore {\it{any}} map between the resolutions $B_{\bullet} \Lambda(N')\rightarrow B_{\bullet} \Lambda(N)$ must be a homotopy equivalence. Indeed any extension to either resolution
of the identity map $E \rightarrow E$ is homotopic to the identity map of complexes, cf. \cite[Thm. 3, p.~141]{GM03}.
\end{proof}

As a result, there is a natural restriction map $\res_{N,N'}: R\Hom_{\Lambda(N)}(E,M) \rightarrow R\Hom_{\Lambda(N')}(E,M)$ defined as the composition
$$
\Hom_{\Lambda(N)}(B_{\bullet} \Lambda(N),M) \overset{\text{inc}}\longrightarrow \Hom_{\Lambda(N')}(B_{\bullet} \Lambda(N),M) \overset{\text{qis}}{\longrightarrow} \Hom_{\Lambda(N')}(B_{\bullet} \Lambda(N'),M).
$$
The first map is just the inclusion of a subcomplex; the second is the quasi-isomorphism (in fact homotopy equivalence) obtained by composing with the map of Lemma \ref{he}.
Taking cohomology $h^i$ yields restriction maps
$$
\res_{N,N'}^i: \Ext^i_{\Lambda(N)}(E,M) \longrightarrow \Ext^i_{\Lambda(N')}(E,M).
$$
These are $K$-equivariant if $M$ is an $E[K]$-module and $N, N' \vartriangleleft K$. As $N$ varies we get a direct system, and passing to the limit gives the stable cohomology groups of $M$, cf. \cite[p.~18]{Koh17}.

\begin{defn}
We introduce the complex $R\Sigma(M):=\varinjlim_N R\Hom_{\Lambda(N)}(E,M)$ and denote its cohomology groups by
$$
\Sigma^i(M):=h^i\big(R\Sigma(M)\big)=\varinjlim_N \Ext^i_{\Lambda(N)}(E,M).
$$
(Note that $\varinjlim_N$ is exact so it commutes with $h^i$.) The limit is over all open subgroups $N \subset G$.
\end{defn}

Suppose $M$ is an object of $\Mod_{\Lambda(G)}^{\text{pc}}$ (so it carries a $G$-action). For each element $g \in G$ the conjugation map
$g(\cdot)g^{-1}$ induces an isomorphism $\Lambda(N)\overset{\sim}{\longrightarrow} \Lambda(gNg^{-1})$. The induced map between the bar resolutions (together with the automorphism $g: M \rightarrow M$) then gives an isomorphism 
$$
R\Hom_{\Lambda(N)}(E,M) \overset{\sim}{\longrightarrow} R\Hom_{\Lambda(gNg^{-1})}(E,M)
$$
which is the identity if $g \in N$. This way $R\Sigma(M)$ becomes a complex of smooth $E[G]$-modules and the $\Sigma^i(\cdot)$
form a $\delta$-functor $\Mod_{\Lambda(G)}^{\text{pc}} \rightarrow \Rep_E^{\infty}(G)$. Note that $\Sigma^0$ is just the functor $M \mapsto M^{\infty}=\bigcup_N M^N$
which associates the subspace of smooth vectors.

Following \cite[Def.~3.12]{Koh17} we define the higher smooth duals $S^i: \Rep_E^{\infty}(G)\rightarrow \Rep_E^{\infty}(G)$ by composing the $\Sigma^i$ with Pontryagin duality $(\cdot)^{\vee}$. 

\begin{defn}
We introduce the complex $S(\pi):=\varinjlim_N R\Hom_{\Lambda(N)}(E,\pi^{\vee})$ and denote its cohomology groups by
$$
S^i(\pi):=h^i\big(S(\pi)\big)=\varinjlim_N \Ext^i_{\Lambda(N)}(E,\pi^{\vee}).
$$
Note that $S^0(\pi)=\Hom_E(\pi,E)^{\infty}$ is the smooth dual of $\pi$ and the $S^i$ form a $\delta$-functor.  
\end{defn}


\section{First reductions}\label{first}

We now specialize to the case $\pi=\ind_K^G V$. Our goal is to make the complex $R\Hom_{\Lambda(N)}(E,I_K^G \check{V})$ more explicit and understand the restriction maps as we shrink $N$. For now we fix an open $N \vartriangleleft K$.

First we decompose $I_K^G \check{V}$ as a $K$-representation using the Cartan decomposition: Let $X_*(S)^+$ be the set of dominant coweights (relative to $B$). That is, the set of cocharacters $\mu \in X_*(S)$ such that $\langle \mu,\alpha \rangle\geq 0$ for all $\alpha \in \Phi^+$. The map $\mu \mapsto \mu(\varpi)$ gives an isomorphism 
$X_*(S) \overset{\sim}{\longrightarrow} S/S(\OO)$ which clearly restricts to a bijection $X_*(S)^+ \overset{\sim}{\longrightarrow} S^+/S(\OO)$ where $S^+$ denotes the monoid
$$
S^+=\{s\in S: |\alpha(s)|\leq 1 \y \forall \alpha \in \Phi^+\}.
$$
The Cartan decomposition reads $G=\bigsqcup_s KsK$ with $s\in S^+$ running over a fixed set of representatives for $S^+/S(\OO)$. (To fix ideas take $s=\mu(\varpi)$ for a choice of uniformizer $\varpi$ and let $\mu \in X_*(S)^+$ vary.)

\begin{defn}
Let $K^s:=K \cap s^{-1}Ks$ act on $\check{V}$ by the rule $\kappa \ast_s \check{v}:=(s\kappa s^{-1})\check{v}$. We denote the resulting representation of $K^s$ by $\check{V}^s$. 
\end{defn}

With this notation we have the following Mackey factorization of $I_K^G \check{V}$.

\begin{lem}
There is an isomorphism of $E[K]$-modules $I_K^G \check{V} \overset{\sim}{\longrightarrow} \prod_s \ind_{K^s}^K \check{V}^s$ given by $f \mapsto (f_s)$ where 
the function $f_s$ is defined by the formula $f_s(\kappa):=f(s\kappa)$.
\end{lem}

\begin{proof}
This is trivial to check; see \cite[p.~39]{Koh17} but beware that our conventions on induced representations are different.
\end{proof}

Consequently,
$$
R\Hom_{\Lambda(N)}(E,I_K^G \check{V}) \overset{\sim}{\longrightarrow} \prod_s R\Hom_{\Lambda(N)}(E,\ind_{K^s}^K \check{V}^s).
$$
Let us fix an $s \in S^+$ and work out the $s$-th factor above. We consider the compact open subgroup $NK^s$ of $K$ (noting that $K^s\subset K$ normalizes $N$)
and do the induction in stages as $\ind_{NK^s}^K \circ \ind_{K^s}^{NK^s}$ by transitivity. This allows us to write
$$
R\Hom_{\Lambda(N)}(E,\ind_{K^s}^K \check{V}^s)=\ind_{NK^s}^K R\Hom_{\Lambda(N)}(E, \ind_{K^s}^{NK^s}\check{V}^s),
$$
where the induction on the right-hand side means we are inducing each term of the complex. Here we have used the following simple observation (with 
$\mathcal{U}=B_{\bullet}\Lambda(N)$ and $\mathcal{V}=\ind_{K^s}^{NK^s}\check{V}^s$):

\begin{lem}
Suppose $\mathcal{V}$ is a $\Lambda(H)$-module where $H \subset K$ is an open subgroup containing $N \vartriangleleft K$. Let $\mathcal{U}$ be a $\Lambda(N)$-module. Then
there is a natural isomorphism
$$
\Hom_{\Lambda(N)}(\mathcal{U}, \ind_H^K\mathcal{V})\overset{\sim}{\longrightarrow} \ind_H^K\Hom_{\Lambda(N)}(\mathcal{U},\mathcal{V})
$$
of $E[K/N]$-modules. (Here $\ind_H^K\mathcal{V}\simeq E[K]\otimes_{E[H]}\mathcal{V}\simeq \Lambda(K)\otimes_{\Lambda(H)}\mathcal{V}$.) 
\end{lem}

\begin{proof}
Sending $\gamma: \mathcal{U} \rightarrow \ind_H^K\mathcal{V}$ to the function $f_{\gamma}: K \rightarrow \Hom_{\Lambda(N)}(\mathcal{U},\mathcal{V})$ defined as
$f_{\gamma}(\kappa)(u)=\gamma(u)(\kappa)$ does the job. (Here $\kappa \in K$ and $u \in \mathcal{U}$.) We leave the details to the reader.
\end{proof}

The restriction map gives an isomorphism of $E[N]$-modules
$\ind_{K^s}^{NK^s}\check{V}^s \overset{\sim}{\longrightarrow} \ind_{N \cap K^s}^N \check{V}^s$. We transfer the $K^s$-action on the source to 
$\ind_{N \cap K^s}^N \check{V}^s$. More explicitly, for a function $f:N \rightarrow \check{V}^s$ in the latter space and $\kappa \in K^s$ we have
$(\kappa f)(n)=\kappa\cdot f(\kappa^{-1}n \kappa)$. Now, by Frobenius reciprocity (and Lemma \ref{he} applied to the inclusion $N\cap K^s \subset N$) we get a quasi-isomorphism
$$
R\Hom_{\Lambda(N)}(E, \ind_{N \cap K^s}^{N}\check{V}^s) \overset{\sim}{\longrightarrow} R\Hom_{\Lambda(N \cap K^s)}(E, \check{V}^s).
$$
$N$ acts trivially on this complex, and the $K^s$-action is the natural one on the target sending an element $\eta: B_{\bullet} \Lambda(N \cap K^s) \rightarrow \check{V}^s$ 
to its conjugate $\kappa \eta\kappa^{-1}$ for $\kappa \in K^s$ (by unwinding the Frobenius reciprocity map). 

\begin{lem}\label{twist}
$R\Hom_{\Lambda(N \cap K^s)}(E, \check{V}^s) \overset{\sim}{\longrightarrow} R\Hom_{\Lambda(s(N \cap K^s)s^{-1})}(E, \check{V})$.
\end{lem}

\begin{proof}
Send an $\eta$ as above to its composition with the map $B_{\bullet}\Lambda(s(N \cap K^s)s^{-1}) \overset{\sim}{\longrightarrow} B_{\bullet}\Lambda(N \cap K^s)$ induced by the conjugation map $s^{-1}(\cdot)s$. One easily checks the resulting map is $\Lambda(s(N \cap K^s)s^{-1})$-linear for the natural action of 
$s(N \cap K^s)s^{-1}=K \cap sNs^{-1}\subset K$ on $\check{V}$.
\end{proof}

\begin{rem}
The target in Lemma \ref{twist} carries a natural action of $K^{s^{-1}}=K \cap sKs^{-1}$ since it contains $s(N \cap K^s)s^{-1}$ as a normal subgroup. The map 
$s^{-1}(\cdot)s$ gives an isomorphism $K^s \overset{\sim}{\longrightarrow} K^{s^{-1}}$ under which the $K^{s^{-1}}$-action on the target corresponds to the 
$K^s$-action on the source.
\end{rem}

We conclude that there is a natural quasi-isomorphism
$$
R\Hom_{\Lambda(N)}(E,I_K^G \check{V}) \overset{\sim}{\longrightarrow} {\prod}_{s\in S^+/S(\OO)} \ind_{NK^s}^K R\Hom_{\Lambda(N \cap K^s)}(E, \check{V}^s)
$$
which is compatible with restriction in the following sense. For an open subgroup $N'\subset N$ the restriction map $\res_{N,N'}$ (with $M=I_K^G \check{V}$) corresponds to
$\prod_s \varrho_s$ where 
$$
\varrho_s: \ind_{NK^s}^K R\Hom_{\Lambda(N \cap K^s)}(E, \check{V}^s) \longrightarrow \ind_{N'K^s}^K R\Hom_{\Lambda(N' \cap K^s)}(E, \check{V}^s)
$$
is composition with the restriction map
$$
\res_{N \cap K^s, N' \cap K^s}: R\Hom_{\Lambda(N \cap K^s)}(E, \check{V}^s) \longrightarrow R\Hom_{\Lambda(N' \cap K^s)}(E, \check{V}^s).
$$


\section{Iwahori factorization and $p$-valuations}

The next step is to understand the individual complexes $R\Hom_{\Lambda(K \cap sNs^{-1})}(E, \check{V})$ for $s\in S^+$ appearing in the factorization of 
$R\Hom_{\Lambda(N)}(E,I_K^G \check{V})$, cf. Lemma \ref{twist}. 

We will now assume $N \vartriangleleft K$ has Iwahori factorization (with respect to $B$). This means three things:

\begin{itemize}
\item[(i)] $(N \cap \bar{U})\times (N \cap T)\times (N\cap U) \overset{\sim}{\longrightarrow} N$ is a homeomorphism;
\item[(ii)] $s(N\cap U)s^{-1}\subset N\cap U$
\item[(iii)] $s(N\cap \bar{U})s^{-1}\supset N\cap \bar{U}$
\end{itemize}
for all $s \in S^+$. It is well-known that $K$ has a neighborhood basis at the identity consisting of such groups. In fact one can take $N=K_r$ for large enough $r$, cf. 
\cite[Prop.~1.4.4]{Cas95} and \cite[Thm.~2.5]{IM65}, or the very readable account \cite[Prop.~3.12]{Rab07}.

In this section we fix an $s \in S^+$. Conjugating (i) by $s$ and then intersecting with $K$ results in a decomposition of the group of interest,
$$
(K \cap sNs^{-1} \cap \bar{U}) \times (N \cap T) \times (K \cap sNs^{-1} \cap U) \overset{\sim}{\longrightarrow}
K \cap sNs^{-1}.
$$
Here the rightmost factor is contained in $N \cap U$ by (ii), and the leftmost factor contains $N \cap \bar{U}$ by (iii). The homeomorphism in (i) is not an isomorphism of groups. Our first goal in this section is to show that it nevertheless becomes a group isomorphism after passing to the graded groups defined by a suitable $p$-valuation. 

Recall that a $p$-valuation on a group $N$ is a function $\omega: N \backslash \{1\} \rightarrow (\frac{1}{p-1},\infty)$ satisfying the axioms

\begin{itemize}
\item $\omega(x^{-1}y)\geq \min\{\omega(x),\omega(y)\}$
\item $\omega([x,y])\geq \omega(x)+\omega(y)$
\item $\omega(x^p)=\omega(x)+1$
\end{itemize}
for all $x,y \in N$. Here the commutator is $[x,y]=xyx^{-1}y^{-1}$. This notion was introduced by Lazard in \cite{Laz65}, and his theory is elegantly exposed in \cite{Sch11} which we will use as our main reference. A $p$-valued group $(N,\omega)$ is said to be saturated if all $x \in N$ satisfying the inequality $\omega(x)> \frac{p}{p-1}$ lie in $N^p$, cf. \cite[p.~187]{Sch11}. If so the set of $p^n$-powers $N^{p^n}$ is a subgroup.

\begin{lem}\label{val}
There are arbitrarily small $N$ which admit a $p$-valuation $\omega$ such that one has an equality
$$
\omega(n)=\min\{\omega(\bar{u}), \omega(t), \omega(u)\}
$$
for all $n \in N$ with Iwahori factorization $n=\bar{u}tu$ as in (i) above, and $(N,\omega)$ is saturated. 
\end{lem}

\begin{proof} 
Consider the smooth affine group scheme $U_{/\OO}$ and let $N=K_r$ for some $r$. Note that 
$$
N \cap U=\ker\big(U(\OO)\rightarrow U(\OO/\varpi^r\OO)\big).
$$
By a general observation of Serre this carries a natural $p$-valuation $\omega_U$ for which $(N \cap U,\omega_U)$ is saturated when $r$ is large enough, see 
\cite[Ch.~IV.9]{Se64} and \cite[Sect. 2.3]{Se95}, or the more recent \cite[Lem.~2.2.2]{HKN} (and the pertaining discussion on p. 239). The construction goes via the formal group law given by choosing coordinates $\hat{U}\simeq \text{Spf}\big(\OO[\![T_1,\ldots,T_\nu]\!]\big)$ for the formal completion of $U$ at the identity, cf. \cite[Sect.~II.10]{Dem86}.

The same comments apply to the group schemes $\bar{U}$ and $T$, which results in the two saturated $p$-valued groups $(N \cap \bar{U}, \omega_{\bar{U}})$ and $(N \cap T, \omega_T)$.  Choosing ordered bases gives homeomorphisms 
$$
\bar{\phi}: N \cap \bar{U} \overset{\sim}{\longrightarrow} \Z_p^a \y \y \y \y \y 
\psi: N \cap T \overset{\sim}{\longrightarrow} \Z_p^b \y \y \y \y \y 
\phi: N \cap U \overset{\sim}{\longrightarrow} \Z_p^a,
$$
all sending $1_N$ to the zero-vector. Composing $\bar{\phi}\times \psi \times \phi$ with the inverse of the multiplication map in (i) gives a global chart
$\Phi: N \overset{\sim}{\longrightarrow} \Z_p^d$ sending $1_N \mapsto 0$, where $d=2a+b$. It follows that $N_m=\Phi^{-1}(p^m \Z_p^d)$ 
is a subgroup for $m>>0$ chosen large enough for the formal group law to have coefficients in $\Z_p$, cf. the proof of 
\cite[Thm.~27.1]{Sch11} which also verifies the following defines a $p$-valuation on $N_m$,
$$
\tilde{\omega}(n)=\delta+\max\{\ell: n \in N_{m+\ell}\}=\delta+\min\{v(x_1),\ldots,v(x_d)\}, \y \y \y \y \y \Phi(n)=(p^mx_1,\ldots,p^mx_d).
$$
Here $\delta=1$ if $p>2$, and $\delta=2$ if $p=2$. Note that the multiplication map 
$$
(N_m \cap \bar{U})\times (N_m \cap T)\times (N_m\cap U) \overset{\sim}{\longrightarrow} N_m
$$ 
is trivially a homeomorphism since its composition with $\Phi$ is the restriction of $\bar{\phi}\times \psi \times \phi$, and for the same reason
$(N_m,\tilde{\omega})$ clearly satisfies the properties in the lemma.
It remains to check properties (ii) and (iii) for $N_m$. We will only do (ii); the argument for (iii) is similar by replacing $s\in S^+$ with $s^{-1}$. For any $u \in N_m \cap U$ we have to check that $sus^{-1}\in N_m$. This follows from the fact that $(N \cap U,\omega_U)$ is saturated, which implies 
$$
(N \cap U)^{p^m}=\phi^{-1}(p^m \Z_p^a)=N_m \cap U
$$ 
by \cite[Cor.~26.12]{Sch11}.
\end{proof}

\section{Review of Lazard theory}\label{laz}

Let $(N,\omega)$ be a $p$-valued group as in the previous lemma \ref{val} of the last section. In this section we review some general constructions and results of Lazard, partly to set up more notation. For any $v\in \R_{>0}$ we let 
$$
N_v=\{n \in N: \omega(n)\geq v\} \y \y \y \y \y N_{v+}=\{n \in N: \omega(n)>v\}. 
$$
Both are normal subgroups of $N$, and $\gr_vN=N_v/N_{v+}$ is a central subgroup of $N/N_{v+}$ (in particular $\gr_vN$ is abelian). We form the associated graded abelian group
$$
\gr N=\bigoplus_{v>0}\gr_vN
$$
which has additional structure. First of all $\gr N$ is an $\F_p$-vector space since $\omega(x^p)=\omega(x)+1$. Moreover, the commutator $[x,y]=xyx^{-1}y^{-1}$
defines a Lie bracket on $\gr N$ which turn it into a graded Lie algebra ($\gr_vN \times \gr_{v'}N \rightarrow \gr_{v+v'}N$). Finally, $\gr N$ naturally becomes a module over the one-variable polynomial ring $\F_p[\pi]$ as follows. The indeterminate $\pi$ acts on $\gr N$ as the degree one map $\pi: \gr_v N \rightarrow \gr_{v+1}N$ given by
$\pi(nN_{v+})=n^pN_{(v+1)+}$ where $n \in N_v$. Since $(N,\omega)$ is of finite rank $\gr N$ is a free $\F_p[\pi]$-module of said rank; which equals $\dim N$. 
We refer to \cite[Sect.~23-25]{Sch11} where all of the above is explained in great detail.

The same remarks apply to the $p$-valued groups $(N \cap \bar{U}, \omega)$ etc., and because of the formula for $\omega$ in Lemma \ref{val} we deduce
that multiplication defines a homeomorphism
$$
(N \cap \bar{U})_v\times (N \cap T)_v\times (N\cap U)_v \overset{\sim}{\longrightarrow} N_v
$$
and similarly for $N_{v+}$. Since $\gr_vN$ is abelian we conclude that there is an isomorphism of $\F_p$-vector spaces
$$
\gr_v (N \cap \bar{U})\oplus \gr_v (N \cap T)\oplus \gr_v(N\cap U) \overset{\sim}{\longrightarrow} \gr_v N,
$$
and summing up over $v$ gives an analogous decomposition of $\gr N$. Further inspection reveals that this direct sum decomposition
\begin{equation}\label{griwa}
\gr (N \cap \bar{U})\oplus \gr (N \cap T)\oplus \gr (N\cap U) \overset{\sim}{\longrightarrow} \gr N
\end{equation}
preserves the $\F_p[\pi]$-module structures, and thus becomes an isomorphism of graded Lie algebras over $\F_p[\pi]$. Here we may assume the Lie bracket is trivial by the perturbation argument given at the very end of Section \ref{formula} below. We will exploit the ensuing factorization of the universal enveloping algebra of $\gr N \otimes_{\F_p[\pi]}E$ below.

Let $\WW=W(E)$ be the ring of Witt vectors (a complete DVR with residue field $E$ in which $p$ remains prime) and consider the completed group algebra $\WW[\![N]\!]$. As 
explained in \cite[Sect.~28]{Sch11} the $p$-valuation $\omega$ defines a function $\tilde{\omega}:  \WW[\![N]\!]\backslash \{0\} \rightarrow \R_{\geq 0}$ which extends $\omega$ in the sense that $\tilde{\omega}(n-1)=\omega(n)$ holds for all $n \in N$. If we fix an ordered basis $(n_1,\ldots,n_d)$ for $(N,\omega)$ it is explicitly given by the formula
$$
\tilde{\omega}(\lambda)=\inf_{\alpha} \bigg(v(c_{\alpha})+\sum_{i=1}^d \alpha_i \omega(n_i)\bigg) \y \y \y \y 
\lambda=\sum_{\alpha}c_{\alpha}\bf{b}^{\alpha}.
$$
Here $\alpha=(\alpha_1,\ldots,\alpha_d)\in \N^d$ and ${\bf{b}}^{\alpha}={\bf{b}}_1^{\alpha_1}\cdots {\bf{b}}_d^{\alpha_d}$ where ${\bf{b}}_i=n_i-1$. We emphasize that the function $\tilde{\omega}$ is independent of the choice of basis however, cf. \cite[Cor.~28.4]{Sch11}. This gives rise to a filtration of $\WW[\![N]\!]$ as follows. For $v \geq 0$ we let
$$
J_v=\WW[\![N]\!]_v=\{\lambda: \tilde{\omega}(\lambda)\geq v\} \y \y \y \y J_{v+}=\WW[\![N]\!]_{v+}=\{\lambda: \tilde{\omega}(\lambda)>v\}.
$$
These two-sided ideals form a fundamental system of open neighborhoods at zero, and can be made very explicit. For instance $\WW[\![N]\!]_v$ is the smallest closed $\WW$-submodule containing all elements of the form $p^{\mu}(\nu_1-1)\cdots(\nu_t-1)$ with $\mu+\sum_{i=1}^t \omega(\nu_i)\geq v$, cf. \cite[Thm.~28.3(ii)]{Sch11}.
The graded completed group algebra is then defined as
$$
\gr \WW[\![N]\!]=\bigoplus_{v \geq 0} \gr_v \WW[\![N]\!], \y \y \y \y \gr_v \WW[\![N]\!]=J_v/J_{v+}.
$$
This is naturally an algebra over $\gr \WW$ (formed with respect to the filtration $p^n\WW$). Note that $\gr \Z_p \simeq \F_p[\pi]$ via the identification
$p+p^2\Z_p\leftrightarrow \pi$, and this is how we view $\gr \WW$ as an $\F_p[\pi]$-algebra below.
For each $v>0$ there is a natural homomorphism $\mathcal{L}_v: \gr_v N \rightarrow \gr_v \WW[\![N]\!]$ sending $nN_{v+}\mapsto (n-1)+J_{v+}$, and one of the main results of Lazard is that $\mathcal{L}=\oplus_{v >0} \mathcal{L}_v$ extends to an isomorphism of graded $\gr \WW$-algebras
$$
\tilde{\mathcal{L}}: \gr \WW \otimes_{\F_p[\pi]}U(\gr N) \overset{\sim}{\longrightarrow} \gr \WW[\![N]\!].
$$
See \cite[Thm.~28.3(i)]{Sch11}. We are not assuming $\omega$ is $\Z$-valued, and this flexibility will be important later when we perturb $\omega$ to make the Lie algebra $\gr N$ abelian, cf. \cite[Lem.~26.13(i)]{Sch11}. However, we may and do assume that $\omega$ takes values in $\frac{1}{A}\Z$ for some $A \in \Z_{>0}$ (cf. \cite[Cor.~33.3]{Sch11}). Then we reindex and let $\Fil^i \WW[\![N]\!]:=J_{\frac{i}{A}}$, which defines a ring filtration of $\WW[\![N]\!]$ indexed by integers $i \geq 0$.

We will employ the analogous results for the reduction $\Lambda(N)=E \otimes_{\WW}\WW[\![N]\!]$. We endow $\Lambda(N)$ with the quotient filtration $\Fil^i \Lambda(N)$ defined as the image of $\Fil^i \WW[\![N]\!]$ under the quotient map $\WW[\![N]\!] \twoheadrightarrow \Lambda(N)$. The associated graded $E$-algebra 
$\gr \Lambda(N)=\oplus_{i\geq 0}\Fil^i \Lambda(N)/\Fil^{i+1} \Lambda(N)$ is isomorphic to $E \otimes_{\gr \WW}\gr \WW[\![N]\!]$. The tensor product $E \otimes_{\gr \WW}\tilde{\mathcal{L}}$ therefore induces an isomorphism
$$
\bar{\mathcal{L}}: U(E \otimes_{\F_p[\pi]} \gr N)=E \otimes_{\F_p[\pi]} U(\gr N) \overset{\sim}{\longrightarrow} \gr \Lambda(N).
$$
Here $\F_p[\pi]\rightarrow E$ takes $\pi\mapsto 0$. If we tensor (\ref{griwa}) by $E$ over $\F_p[\pi]$ and take universal enveloping algebras the Kronecker product gives an isomorphism
$$
U(E \otimes_{\F_p[\pi]}\gr (N \cap \bar{U}))\otimes_E U(E \otimes_{\F_p[\pi]}\gr (N \cap T))\otimes_E U(E \otimes_{\F_p[\pi]}\gr (N\cap U)) \overset{\sim}{\longrightarrow} U(E \otimes_{\F_p[\pi]}\gr N)
$$
which via $\bar{\mathcal{L}}$ gives the main take-away from this section; namely that there is a natural isomorphism of graded $E$-algebras
\begin{equation}\label{lamb}
\gr \Lambda(N \cap \bar{U}) \otimes_E \gr \Lambda(N \cap T) \otimes_E \gr \Lambda(N \cap U) \overset{\sim}{\longrightarrow} \gr \Lambda(N).
\end{equation}
In the next section we will extend this to $s$-conjugates ($s\in S^+$) and invoke a K\"{u}nneth formula. 

\section{Application of the K\"{u}nneth formula}\label{formula}

The isomorphism (\ref{lamb}) extends easily to arbitrary $s \in S^+$. Having fixed $(N,\omega)$ as in Lemma \ref{val} we define a $p$-valuation $\omega_s$ on 
$K \cap sNs^{-1}$ by the formula $\omega_s(sns^{-1})=\omega(n)$. It is compatible with the Iwahori factorization of $K \cap sNs^{-1}$ as in Lemma \ref{val}. Defining $\gr \Lambda(K \cap sNs^{-1})$ and so on relative to $\omega_s$ the arguments leading up to (\ref{lamb}) therefore yield more generally an isomorphism of graded $E$-algebras
$$
\gr \Lambda(K \cap sNs^{-1} \cap \bar{U}) \otimes_E \gr \Lambda(N \cap T) \otimes_E \gr \Lambda(K \cap sNs^{-1} \cap U) \overset{\sim}{\longrightarrow} \gr \Lambda(K \cap sNs^{-1}).
$$
For $s=1$ one recovers (\ref{lamb}). The shuffle product (cf. \cite[Prop.~4.2.4]{Lod98} for example) gives a homotopy equivalence relating bar resolutions 
$$
B_{\bullet}\gr \Lambda(K \cap sNs^{-1} \cap \bar{U}) \otimes_E B_{\bullet}\gr \Lambda(N \cap T) \otimes_E B_{\bullet}\gr \Lambda(K \cap sNs^{-1} \cap U) \overset{}{\longrightarrow} B_{\bullet}\gr \Lambda(K \cap sNs^{-1}),
$$
and in turn we have a K\"{u}nneth formula in the form of a quasi-isomorphism 
\begin{multline}
R\Hom_{\gr \Lambda(K \cap sNs^{-1})}(E,\gr \check{V}) \overset{\sim}{\longrightarrow} \\
R\Hom_{\gr \Lambda(K \cap sNs^{-1} \cap \bar{U})}(E,\gr \check{V}) \otimes_E 
R\Hom_{\gr \Lambda(N \cap T)}(E,E) \otimes_E 
R\Hom_{\gr \Lambda(K \cap sNs^{-1} \cap U)}(E,E).
\end{multline}
Here $\gr \check{V}$ is defined with respect to $\omega_s$. That is we first filter $\check{V}$ by $\Fil^i \check{V}:=\Fil^i \Lambda(K \cap sNs^{-1})\check{V}$ and let $\gr \check{V}$ be the associated graded module. We have taken $N$ small enough that it acts trivially on $V$ and therefore $\gr\check{V}$ factors as an external tensor product $\gr\check{V} \boxtimes E \boxtimes E$ where the two $E$'s denote the trivial modules over $\gr \Lambda(N \cap T)$ and $\gr \Lambda(K \cap sNs^{-1} \cap U)$ respectively. Note that
$K \cap sNs^{-1} \cap U\subset N$ by property (ii) of an Iwahori factorization. Taking cohomology $h^i$ results in $E$-vector space isomorphisms 
\begin{multline}\label{kun}
\Ext^i_{\gr \Lambda(K \cap sNs^{-1})}(E,\gr \check{V}) \overset{\sim}{\longrightarrow} \\
\bigoplus_{a+b+c=i}
\Ext^a_{\gr \Lambda(K \cap sNs^{-1} \cap \bar{U})}(E,\gr \check{V}) \otimes_E 
\Ext^b_{\gr \Lambda(N \cap T)}(E,E) \otimes_E 
\Ext^c_{\gr \Lambda(K \cap sNs^{-1} \cap U)}(E,E)
\end{multline}
compatible with restriction maps when shrinking $N$. Note that by perturbing $\omega$ (i.e., replacing it by $\omega_C=\omega-C$ for sufficiently small $C>0$) we may assume 
all the graded algebras above are polynomial rings over $E$ in a number of variables, cf. \cite[Lem.~26.13]{Sch11}. This will allow us to control some of the factors in the K\"{u}nneth formula using Koszul duality. 

\section{Restriction and the Koszul dual}

We first deal with the two factors $\Ext^b_{\gr \Lambda(N \cap T)}(E,E)$ and $\Ext^c_{\gr \Lambda(K \cap sNs^{-1} \cap U)}(E,E)$ in the K\"{u}nneth formula (\ref{kun}). This can be done uniformly so in this section we let $H$ denote one of the two groups $N \cap T$ or $K \cap sNs^{-1} \cap U$ equipped with the $p$-valuation $\omega_s$. (Note that 
$\omega_s=\omega$ in the case $H=N \cap T$.) We assume that $\omega$ has been chosen in such a way that $\gr \Lambda(H)$ is a polynomial ring over $E$ in a number of variables, or more canonically a symmetric algebra
$$
\gr \Lambda(H) \simeq S(E \otimes_{\F_p[\pi]} \gr H)
$$
via the mod $p$ Lazard isomorphism $\bar{\mathcal{L}}$ discussed in section \ref{laz}. This can be ensured by perturbing $\omega$ if necessary as noted above. Then 
the Yoneda algebra $\bigoplus_{j\geq 0}\Ext_{\gr \Lambda(H)}^{j}(E,E)$ is the Koszul dual which in this case is simply the exterior algebra 
$\bigwedge^{\bullet}(E \otimes_{\F_p[\pi]} \gr H)^*$ of the $E$-linear dual, cf. \cite[Thm.~1.2.5]{BGS96}. In particular
$$
\Ext_{\gr \Lambda(H)}^{j}(E,E) \simeq \bigwedge^{j}(E \otimes_{\F_p[\pi]} \gr H)^*.
$$
For any choice of $H$ we get a saturated $p$-valued group $(H,\omega_s)$, cf. the proof of lemma \ref{val}. In particular the set of $p^m$-powers
$H^{p^m}=H_{(m+\frac{1}{p-1})+}$ is a subgroup, and in fact they form a fundamental system of open neighborhoods of the identity as $m$ varies, cf. \cite[Prop.~26.15]{Sch11}.

\begin{lem}\label{ext}
Let $n \geq 1$. Then $\Ext_{\gr \Lambda(H^{p^m})}^n(E,E)\overset{0}{\longrightarrow} \Ext_{\gr \Lambda(H^{p^{m+1}})}^n(E,E)$ for all $m \in \Z_{\geq 0}$.
\end{lem}

\begin{proof}
One reduces to the case $m=0$ by replacing $H$ by $H^{p^m}$. We then have to show the vanishing of 
$$
\bigwedge^{n}(E \otimes_{\F_p[\pi]} \gr H)^* \longrightarrow \bigwedge^{n}(E \otimes_{\F_p[\pi]} \gr H^p)^* 
$$
for $n>0$ and we clearly may assume that $n=1$. In other words, we are to check the vanishing of the dual map 
$$
E \otimes_{\F_p[\pi]} \gr H^p=\gr H^p/\pi\gr H^p \longrightarrow \gr H/\pi\gr H=E \otimes_{\F_p[\pi]} \gr H.
$$
This is trivial to verify. Indeed $\gr_{v+1} H^p=\pi \gr_v H$ holds for all $v>0$ as follows straight from the definition of the $\pi$-action on $\gr H$, which proves the claim. 
\end{proof}

The above result is the key to establishing hypothesis (\ref{van}) in the Appendix.

\section{Invoking the appendix}

We will apply Theorem \ref{koz} of Appendix A to the decreasing chain of subalgebras of $\Lambda(K \cap sNs^{-1})$ given by the $p^m$-powers of $N$,
$$
A^{(m)}:=\Lambda\big(K \cap sN^{p^m}s^{-1}\big).
$$
Note that by the proof of lemma \ref{val} we know that $N^{p^m}$ inherits an Iwahori factorization, and therefore
$$
(K \cap sN^{p^m}s^{-1} \cap \bar{U}) \times (N \cap T)^{p^m} \times (sNs^{-1} \cap U)^{p^m} \overset{\sim}{\longrightarrow}
K \cap sN^{p^m}s^{-1}.
$$
Here we have used that both $N \cap T$ and $N \cap U$ are saturated to move the $p^m$-powers outside. (Indeed, suppose $x \in sN^{p^m}s^{-1}\cap U$. Then $s^{-1}xs \in N^{p^m}\cap U=\joinrel=\joinrel=(N \cap U)^{p^m}$ -- as $N \cap U$ is saturated. We conclude that 
$x \in (sNs^{-1} \cap U)^{p^m}$ since $s\in S^+\subset T$ normalizes the unipotent radical $U$.)
Also, in the third factor we have deliberately written
$sNs^{-1} \cap U$ instead of $K \cap sNs^{-1} \cap U$ (they are the same since $s(N \cap U)s^{-1}$ lies in $N \cap U$ and therefore in $K$). The upshot is 
the K\"{u}nneth formula (\ref{kun}) also applies to each of the graded algebras $\gr A^{(m)}$. 

\begin{lem}\label{achk}
Hypothesis (\ref{van}) of Appendix A is fulfilled for all $n >\dim U$. That is, the restriction map
$$
\Ext_{\gr A^{(m)}}^n(E, \gr \check{V}) \longrightarrow \Ext_{\gr A^{(m+1)}}^n(E, \gr \check{V})
$$
vanishes for all $m$. 
\end{lem}

\begin{proof}
To make the argument more transparent we will only give the details for $m=0$. In the K\"{u}nneth formula (\ref{kun}) for $\Ext_{\gr A^{(0)}}^n(E, \gr \check{V})$
we consider the restriction map to $A^{(1)}$ on the $(a,b,c)$-summand. When $b>0$ or $c>0$ we get zero by lemma \ref{ext}. What is left is to see what happens to restriction on the $(n,0,0)$-summand
$$
\Ext^n_{\gr \Lambda(K \cap sNs^{-1} \cap \bar{U})}(E,\gr \check{V}).
$$
However, this summand itself is zero when $n>\text{rank}(K \cap sNs^{-1} \cap \bar{U})=\dim \bar{U}=\dim U$, cf. \cite[V.2.2]{Laz65} and \cite[Thm.~27.1]{Sch11}.
\end{proof}

Theorem \ref{koz} applies and yields the following key result.

\begin{prop}\label{dimu}
Let $n>\dim U$ be arbitrary. Then the restriction map 
$$
\Ext_{\Lambda(K \cap sNs^{-1})}^n (E, \check{V}) \overset{}{\longrightarrow} \Ext_{\Lambda (K \cap sN^{p^m}s^{-1})}^n (E, \check{V})
$$ 
vanishes for all $m>\text{amp}(\check{V})+n$. (Here $\text{amp}(\cdot)$ is the amplitude introduced in the Appendix.)
\end{prop}

\begin{proof}
We may arrange for $\gr \Lambda(K \cap sNs^{-1})$ to be Koszul (e.g., a polynomial algebra) by perturbing $\omega$ if necessary. We checked in lemma \ref{achk} that hypothesis 
(\ref{van}) is satisfied for $n>\dim U$ so Theorem \ref{koz} applies and gives the vanishing of $\Ext_{A}^n(E, \check{V})\rightarrow \Ext_{A^{(m)}}^n(E, \check{V})$ as long as 
$m>\text{amp}(\check{V})+n$. Here we use the notation from the Appendix. In particular $A=A^{(0)}$, cf. the paragraph containing (\ref{van}). 
\end{proof}

\section{Proof of the main result}

By Lemma \ref{twist} we may reformulate Proposition \ref{dimu} as saying that the map
$$
\Ext_{\Lambda(N \cap K^s)}^n(E,\check{V}^s) \longrightarrow \Ext_{\Lambda(N^{p^m} \cap K^s)}^n(E,\check{V}^s) 
$$
vanishes for $n>\dim U$ and $m>\text{amp}(\check{V}^s)+n$. The amplitude $\text{amp}(\check{V})$ in Proposition \ref{dimu} is computed relative to $\omega_s$, and therefore
coincides with $\text{amp}(\check{V}^s)$ which is relative to the $\omega$ from $N$. Note that all our $\Lambda$'s satisfy $\Fil^0\Lambda=\Lambda$ so $\nu\geq 0$ (in the notation of the appendix); in other words the amplitude $\text{amp}(M)$ is at most $\mu$ (the first index for which $\Fil^{\mu}M=0$). 

\begin{lem}\label{unif}
$\text{amp}(\check{V}^s)$ is uniformly bounded in $s \in S^+$.
\end{lem}

\begin{proof}
We recall that $\text{amp}(\check{V}^s)$ is relative to the filtration $\Fil^i\check{V}^s:=\Fil^i\Lambda(N \cap K^s)\check{V}$ where the filtration of 
$\Lambda(N \cap K^s)$ is defined with respect to (the restriction of) $\omega$. As $N \cap K^s \subset N$ it is enough to observe that
$\Fil^i \Lambda(N)\check{V}=0$ for $i\geq A$ for some $A=A_{N,V}>0$ depending only on $N$ and $V$ (not $s$); where again  
$\Fil^i \Lambda(N)$ is relative to $\omega$. Assuming $N$ is pro-$p$ the vanishing for $i \geq A$ follows from Nakayama's lemma by comparing 
$\Fil^i \Lambda(N)\check{V}$ to the filtration $\frak{m}_{\Lambda(N)}^i\check{V}$ which must be stationary since $V$ is finite-dimensional, 
cf. \cite[Rem.~28.1]{Sch11}.
\end{proof}

From the last paragraph of section \ref{first} we have the isomorphism
$$
\Ext^n_{\Lambda(N)}(E,I_K^G \check{V}) \overset{\sim}{\longrightarrow} {\prod}_{s\in S^+/S(\OO)} \ind_{NK^s}^K \Ext^n_{\Lambda(N \cap K^s)}(E, \check{V}^s)
$$
which is compatible with the restriction maps on either side. We conclude that
$$
\Ext^n_{\Lambda(N)}(E,I_K^G \check{V}) \overset{}{\longrightarrow} \Ext^n_{\Lambda(N^{p^m})}(E,I_K^G \check{V})
$$
vanishes for $n>\dim U$ and $m>A_{V,N}+n$ arbitrary (here $A_{V,N}$ is the uniform bound for 
$\text{amp}(\check{V}^s)$ found in the proof of lemma \ref{unif}). In particular this proves the vanishing of $S^n(\ind_K^GV)$ for $n>\dim U$, which is our main Theorem \ref{main} from the introduction.

\begin{rem}\label{trivial}
For the trivial weight $V=E$ one can strengthen this significantly. Indeed $\forall n>0$ the map
$$
\Ext_{\Lambda(K \cap sNs^{-1})}^n (E,E) \overset{}{\longrightarrow} \Ext_{\Lambda (K \cap sN^{p}s^{-1})}^n (E,E)
$$ 
vanishes for $N$ small. Thus $S^n(\ind_K^G1)=0$ for $n>0$ as mentioned in the introduction. The reason is that
one can arrange for $K \cap sNs^{-1}$ to be equi-$p$-valued (when equipped with the valuation $\omega_s$) by shrinking $N$ and therefore by Lazard's computation of its mod $p$ cohomology algebra \cite[p.~183]{Laz65}
we have 
$$
\Ext_{\Lambda(K \cap sNs^{-1})}^n (E,E)=H^n(K \cap sNs^{-1},E)=\joinrel=\joinrel=\bigwedge^n \Hom(K \cap sNs^{-1},E).
$$
By saturation elements of $\Hom(K \cap sNs^{-1},E)$ vanish upon restriction to $K \cap sN^{p}s^{-1}$.
\end{rem}

\section{A few unaddressed questions}\label{nv}

In this section we sharpen the expectation that $S^{\dim(G/B)}(\ind_K^G V)$ is non-trivial for sufficiently non-degenerate weights $V$. We first recall the notion of regularity introduced in \cite{HV12}. For a weight $V$ one defines $M_V$ to be the (unique) largest standard Levi subgroup for which $M_V(k)$ preserves the line $V^{U(k)}$. 

\begin{defn}
Let $P \supset B$ be a parabolic subgroup defined over $F$ with standard Levi factor $M$. We say that $V$ is $M$-regular if $M_V \subset M$.
\end{defn}

For instance, all weights are $G$-regular. For $\GL_2$ the weight $V$ is $T$-regular exactly when $\dim V>1$. We believe $T$-regularity is enough to guarantee non-vanishing, but we have no evidence. 

\begin{quest}\label{qone}
Is $S^{\dim(G/B)}(\ind_K^G V)\neq 0$ for all $T$-regular weights $V$?
\end{quest}

We have not been able to show this even in the case of $\GL_2(\Q_p)$. Using \cite[Thm.~1.2]{HV12} and \cite[Thm.~4.7(ii)]{Koh17} one can at least prove the non-vanishing of 
$S^{\dim(G/B)}$ on principal series representations $\ind_K^G V \otimes_{\mathcal{H}_G(V),\chi} E$ where the eigensystems $\chi: \mathcal{H}_G(V)\rightarrow E$ factor through the Satake homomorphism $\mathcal{H}_G(V) \longrightarrow \mathcal{H}_T(V_{\bar{U}(k)})$. 

Initially we had hoped to prove the following more precise bound in Theorem \ref{main}: Suppose $P\supset B$ is a standard parabolic subgroup defined over $F$ with standard Levi factor $M_P\supset T$ and unipotent radical $U_P$. Let $V$ be a weight for which $M_P(k)$ preserves the line $V^{U(k)}$. In other words $M_P \subset M_V$. 

\begin{quest}\label{qtwo}
Is it true that $S^i(\ind_K^G V)=0$ for all $i>\dim(U_P)$?
\end{quest}

This is consistent with our results for $P=G$ and $P=B$. What seems to go wrong if one tries to mimic our argument is that in the Iwahori factorization of $N$ relative to $P$, the middle factor $N \cap M_P$ is not fixed under conjugation by $s \in S^+$. 

If question \ref{qtwo} has a positive answer, it seems natural to strengthen question \ref{qone} as follows. 

\begin{quest}\label{qthree}
Is $S^{\dim(U_P)}(\ind_K^G V)\neq 0$ for weights $V$ with $M_V=M_P$?
\end{quest}

This more refined question was brought to our attention by Niccolo' Ronchetti.



\appendix

\section{A spectral sequence for $\Ext$ over filtered rings}

Let $k$ be a field\footnote{We apologize for the change of notation. Now $k$ is the coefficient field (formerly $E$); not the residue field of $F$. }, and $A$ an augmented $k$-algebra with augmentation map $\epsilon:A \rightarrow k$. We assume $A$ is filtered; meaning it comes with a decreasing filtration 
$A=\Fil^0A \supset \Fil^1A \supset \cdots$ by two-sided ideals $\Fil^iA$ which satisfy the following two properties:

\begin{itemize}
\item[(1)] $\Fil^iA \times \Fil^jA \rightarrow \Fil^{i+j}A$ for all $i,j\in \N$;
\item[(2)] $\epsilon(\Fil^1A)=0$.
\end{itemize}
The associated graded $k$-algebra $\gr A=\oplus_{i\geq 0} \gr^i A=\oplus_{i\geq 0} \Fil^iA/\Fil^{i+1}A$ inherits an augmentation map 
$\epsilon$ by projecting to the first component $A/\Fil^1A$. Note that a filtration-preserving morphism of augmented $k$-algebras
$A'\rightarrow A$ induces a morphism of graded augmented $k$-algebras $\gr A' \rightarrow \gr A$.

\subsection{The filtered bar resolution}

Below we will study the functorial properties of $\Ext_A^i(k,-)$ as we vary $A$. For that we will need a projective resolution of $k$ which is functorial in $A$; the bar resolution
$B_{\bullet}A$,
$$
\cdots \overset{d_3}\longrightarrow A \otimes_k A \otimes_k A \overset{d_2}\longrightarrow A \otimes_k A \overset{d_1}{\longrightarrow} A \overset{\epsilon}{\longrightarrow} k \longrightarrow 0.
$$
Let $B_nA=A^{\otimes(n+1)}$ be the $(n+1)$-fold tensor product, and define the differentials by letting $d_0=\epsilon$ and for $n>0$ let $d_n: B_nA \rightarrow B_{n-1}A$ 
be the map which takes $a_0 \otimes a_1 \otimes \cdots \otimes a_n$ to 
$$
(-1)^n(a_0\otimes \cdots \otimes a_{n-1})\epsilon(a_n)+\sum_{i=0}^{n-1} (-1)^i (a_0\otimes\cdots\otimes a_ia_{i+1}\otimes \cdots \otimes a_n).
$$
Note that $B_nA$ is a free left $A$-module (via the first factor $a_0$) of rank $\dim_k A^{\otimes n}$. It is easily checked that $d_{n-1}\circ d_n=0$, and in fact the complex 
is exact (taking  $a_0 \otimes a_1 \otimes \cdots \otimes a_n$ to $1 \otimes a_0 \otimes a_1 \otimes \cdots \otimes a_n$ defines a contracting homotopy). We conclude that $B_{\bullet}A$
is a resolution of $k$ by free left $A$-modules (usually of infinite rank) which is functorial in $A$; a morphism of augmented $k$-algebras $A' \rightarrow A$ induces a 
morphism of complexes $B_{\bullet}A' \rightarrow B_{\bullet}A$ in the obvious way. This is the reason we choose to work with the bar resolution throughout, as opposed to any projective resolution of $k$.

Since our algebra $A$ is filtered, $B_{\bullet}A$ is filtered as well by endowing each $B_nA$ with the tensor product filtration. That is
$$
\Fil^i B_nA={\sum}_{i_0+i_1+\cdots+i_n=i}\Fil^{i_0}A \otimes_k \Fil^{i_1}A \otimes_k \cdots \otimes_k \Fil^{i_n}A.
$$
This is a (not necessarily free) $A$-submodule. It is trivial to check that $d_n: \Fil^i B_nA \rightarrow \Fil^i B_{n-1}A$, i.e. $B_{\bullet}A$ is a filtered complex
(if $i_n>0$ we have $\epsilon(a_n)=0$, and otherwise $i_0+\cdots +i_{n-1}=i$). Let $\gr B_{\bullet}A$ be the associated graded complex whose $n$th term is 
$\gr B_nA=\oplus_{i\geq 0} \gr^i B_nA$. The next result identifies it with the bar resolution for $\gr A$ which is graded in a natural way. 

\begin{lem}\label{bgr}
$\gr B_{\bullet}A=B_{\bullet}(\gr A)$; more precisely for any $i \geq 0$ the natural map
$$
\gr^i B_n(\gr A)={\bigoplus}_{i_0+i_1+\cdots+i_n=i} \gr^{i_0}A \otimes_k \gr^{i_1}A \otimes_k \cdots \otimes_k \gr^{i_n}A \longrightarrow \gr^i B_nA
$$
is an isomorphism of left $\gr A$-modules commuting with the differentials.
\end{lem}

\begin{proof}
For each $i$ choose a $k$-subspace $\Delta^i \subset \Fil^i A$  such that $\Fil^i A=\Delta^i \oplus \Fil^{i+1}A$. For the rest of the proof fix an $i \geq 0$ and decompose
$A$ as a direct sum of subspaces $A=\Delta^0\oplus\cdots \oplus \Delta^i\oplus \Fil^{i+1}A$. We introduce a more uniform notation 
by letting $D^j \subset \Fil^jA$ denote $\Delta^j$ when $j\leq i$ and $D^{i+1}=\Fil^{i+1}A$. Then 
$B_nA$ decomposes as
$$
B_nA=\big({\bigoplus}_{i_0+\cdots+i_n\leq i }\Delta^{i_0}\otimes_k\cdots \otimes_k \Delta^{i_n}\big)\oplus \big({\bigoplus}_{i_0+\cdots+i_n> i }'D^{i_0}\otimes_k\cdots \otimes_k D^{i_n}\big)
$$
where the prime in the second $\bigoplus'$ indicates we are only summing over $i_0,\ldots,i_n\leq i+1$ with sum $>i$. Clearly this sum $\bigoplus'$
is contained in $\Fil^{i+1}B_nA$. Conversely, noting that $\Fil^jA=\Delta^j \oplus \cdots \oplus \Delta^i\oplus \Fil^{i+1}A$ for $j\leq i+1$ it follows immediately that
$\Fil^{i+1}B_nA$ lies in $\bigoplus'$. Indeed, if $j_0+\cdots +j_n=i+1$,
$$
\Fil^{j_0}A \otimes_k \cdots \otimes_k \Fil^{j_n}A={\bigoplus}_{j_0\leq i_0\leq i+1,\ldots, j_n\leq i_n\leq i+1} D^{i_0}\otimes_k \cdots \otimes_k D^{i_n},
$$
and $i_0+\cdots+i_n\geq i+1$ for $i$'s in that range. We conclude that for each $i\geq 0$ we have
$$
B_nA=\big({\bigoplus}_{i_0+\cdots+i_n\leq i }\Delta^{i_0}\otimes_k\cdots \otimes_k \Delta^{i_n}\big)\oplus \Fil^{i+1}B_nA.
$$
In particular $\gr^i B_nA$ is the direct sum over $i_0+\cdots+i_n=i$. Since $\Delta^i \simeq \gr^iA$ this finishes the proof. 
\end{proof}

Ronchetti pointed me to the reference \cite{Sjo74} which contains results of the same flavor. For instance our lemma \ref{bgr} above appears to be closely related to \cite[Lem.~10]{Sjo74}. Furthermore our lemma \ref{grhom} below is exactly \cite[Lem.~16]{Sjo74} -- in a different notation. Instead of comparing notations we found it easier and more convenient to just include the proofs in this Appendix.  

\subsection{Graded $R\Hom$}

For a left $A$-module $M$ we consider the complex $R\Hom_A(k,M)=\Hom_A(B_{\bullet}A,M)$ of $k$-vector spaces (of $A$-modules if $A$ is commutative). More precisely
$$
\cdots \longrightarrow 0 \longrightarrow \Hom_A(A,M) \overset{\partial_1}{\longrightarrow} \Hom_A(A\otimes_kA,M) \overset{\partial_2}{\longrightarrow} \Hom_A(A\otimes_kA \otimes_k A,M) \overset{\partial_3}{\longrightarrow} \cdots
$$
whose $i$th cohomology group is $\Ext_A^i(k,M)$. Observe that a $k$-algebra map $A' \rightarrow A$ gives a morphism of complexes
$R\Hom_A(k,M)\rightarrow R\Hom_{A'}(k,M)$ which we will refer to as the restriction map (at least when $A'\subset A$ is a subalgebra). Taking cohomology yields maps
$\Ext_A^i(k,M)\rightarrow \Ext_{A'}^i(k,M)$ for each $i$.

Now suppose $M$ is equipped with a decreasing filtration by $A$-submodules $\Fil^iM \supset \Fil^{i+1}M\supset \cdots$ such that
$\Fil^iA \times \Fil^jM \rightarrow \Fil^{i+j}M$ for all $i \in \N$ and $j \in \Z$ (for more flexibility we allow filtrations of $M$ to be indexed by $\Z$); we will often take 
$\Fil^iM:=(\Fil^iA)M$ but this restriction is unnecessary. We let $\gr M=\oplus_{i\in\Z}\gr^iM$ be the associated graded $\gr A$-module.

In this situation $R\Hom_A(k,M)$ becomes a filtered complex by defining ($\forall s \in \Z$)
$$
\Fil^s \Hom_A(B_nA,M):=\{\phi\in \Hom_A(B_nA,M): \phi(\Fil^i B_nA)\subset \Fil^{i+s}M \y \forall i\geq 0\}.
$$
(This is a decreasing filtration by $k$-subspaces, compatible with the differentials $\partial$.) We will always assume the filtration on $M$ satisfies $\Fil^iM=0$ for $i>>0$ and
$\Fil^iM=M$ for $i<<0$. Then clearly also $\Fil^s \Hom_A(B_nA,M)=0$ for $s>>0$; however it may not be exhaustive. We let
$$
{^*\Hom}_A(B_nA,M):=\bigcup_{s\in \Z}\Fil^s \Hom_A(B_nA,M)=\{\phi: \phi(\Fil^i B_nA)=0 \y \forall i>>0\}.
$$
This is all of $\Hom_A(B_nA,M)$ if the filtration on $A$ is finite, i.e. $\Fil^iA=0$ for $i$ sufficiently large. Similarly, we let
$$
R\Hom_{\gr A}(k,\gr M):=\Hom_{\gr A}(B_{\bullet}(\gr A),\gr M)\simeq \Hom_{\gr A}(\gr B_{\bullet} A, \gr M),
$$
with $i$th cohomology group $\Ext_{\gr A}^i (k,\gr M)$. For a fixed $n\geq 0$ and $s \in \Z$ we consider the subspace of homogeneous degree $s$ maps
$$
\Hom_{\gr A}^s(\gr B_n A, \gr M):=\{\psi\in \Hom_{\gr A}(\gr B_n A, \gr M): \psi(\gr^i B_nA)\subset \gr^{i+s}M \y \forall i\geq 0\}.
$$ 
They clearly form a direct sum in $\Hom_{\gr A}(\gr B_n A, \gr M)$ which we will denote by
$$
{^*\Hom}_{\gr A}(\gr B_n A, \gr M):=\bigoplus_{s\in \Z} \Hom_{\gr A}^s(\gr B_n A, \gr M)=\{\psi: \psi(\gr^i B_nA)=0 \y \forall i>>0\}.
$$
(For the inclusion $\supset$ write $\psi=\sum\psi_s$ where $\psi_s$ is defined as the projection $\pi_{i+s}\circ \psi$ on $\gr^i B_nA$. One easily checks $\psi_s$ is $\gr A$-linear
and the vanishing condition on $\psi$ guarantees that $\psi_s=0$ for $|s|>>0$.)

\begin{lem}\label{grhom}
For every $s \in \Z$ there is a natural isomorphism of $k$-vector spaces
$$
\gr^s \Hom_A(B_nA,M) \overset{\sim}{\longrightarrow} \Hom_{\gr A}^s(\gr B_n A, \gr M).
$$
\end{lem}

\begin{proof}
Any $\phi \in \Fil^s \Hom_A(B_nA,M)$ defines a $\psi \in \Hom_{\gr A}^s(\gr B_n A, \gr M)$ in the obvious way; and $\psi=0$ exactly when $\phi\in 
\Fil^{s+1} \Hom_A(B_nA,M)$. In other words, there is a $k$-linear injection 
$$
\gr^s \Hom_A(B_nA,M) \hookrightarrow \Hom_{\gr A}^s(\gr B_n A, \gr M).
$$
To show this is surjective observe that $B_nA=A \otimes_k V$ where $V$ is a filtered $k$-vector space ($A^{\otimes n}$) such that $\Fil^i B_nA=\sum_{p+q=i}F^pA \otimes_k \Fil^qV$ for all $i \geq 0$. Therefore, as filtered $k$-vector spaces,
$$
\Hom_A(B_nA,M)\simeq \Hom_k(V,M).
$$
Similarly $\gr B_nA=B_n(\gr A)=\gr A \otimes_k \gr V$ as graded $k$-vector spaces by Lemma \ref{bgr}. Thus
$$
\Hom_{\gr A}(\gr B_n A, \gr M) \simeq \Hom_k(\gr V, \gr M),
$$
and it remains to observe that $\gr^s \Hom_k(V,M) \hookrightarrow \Hom_k^s(\gr V, \gr M)$ is an isomorphism. Indeed, as in the proof of \ref{bgr}, we may choose a sequence of $k$-subspaces $\nabla^i\subset \Fil^i V$ such that $V=\nabla^0\oplus \cdots \oplus \nabla^i \oplus \Fil^{i+1} V$. Let $\psi \in \Hom_k^s(\gr V, \gr M)$. Since the filtration on $M$ is assumed to be finite, we may choose $i$ large enough that $\psi(\gr^{i+1}V)=0$. Now define a preimage $\phi\in \Fil^s \Hom_k(V,M)$ of $\psi$ as follows. Let $\phi=0$ on $\Fil^{i+1} V$,
and on $\nabla^j\simeq \gr^j V$ for $j=0,\ldots,i$ take $\phi$ to be a lift $\nabla^j \rightarrow \Fil^{j+s}M$ of the given $\psi: \gr^j V \rightarrow \gr^{j+s}M$.
\end{proof}

\subsection{A topological variant}

We now assume $A$ is a pseudocompact $k$-algebra. More precisely, we assume the ideals 
$\Fil^i A$ are open and form a neighborhood basis at $0$, and $A \overset{\sim}{\longrightarrow} \varprojlim A/\Fil^i A$ as topological rings with Artinian (discrete) quotients
$A/\Fil^i A$. We let $\CC_A$ denote the abelian category of pseudocompact $A$-modules (i.e., complete Hausdorff topological left $A$-modules which are inverse limits of discrete finite length $A$-modules). As is well-known, $\CC_A$ has exact inverse limits and enough projectives. Similarly we let $\CC_{\gr A}$ denote the abelian category of $\Z$-graded $\gr A$-modules, with morphisms $\Hom_{\CC_{\gr A}}={^*\Hom}_{\gr A}$ (the sums of homogeneous maps as above). Any projective $\gr A$-module with a $\Z$-grading is projective in $\CC_{\gr A}$ and vice versa (cf. \cite[p.~289]{FF74} for commutative rings). 

The bar resolution has a topological variant $\hat{B}_{\bullet}A$ obtained by replacing $\otimes_k$ with completed tensor products $\hat{\otimes}_k$ everywhere.
For instance $\hat{B}_1A=A \hat{\otimes}_k A=\varprojlim_{i,j}A/\Fil^i A \otimes_k A/\Fil^j A$, cf. \cite[p.~446]{Bru66}. More generally 
$\hat{B}_nA=A^{\hat{\otimes}(n+1)}$. The differential $d_n$ extends to $\hat{B}_nA$ by continuity and this defines a resolution of $k$ in $\CC_A$,
$$
\cdots \overset{d_3}\longrightarrow A\hat{ \otimes}_k A \hat{\otimes}_k A \overset{d_2}\longrightarrow A \hat{\otimes}_k A \overset{d_1}{\longrightarrow} A \overset{\epsilon}{\longrightarrow} k \longrightarrow 0.
$$
Moreover, $\hat{B}_nA$ is projective in $\CC_A$ by \cite[Cor.~3.3]{Bru66} (being an inverse limit of free $A/\Fil^i A$-modules). 

Furthermore, suppose $M$ is a discrete finite length $A$-module with exhaustive and separated filtration $\Fil^i M$ as above.

\begin{lem}\label{starhom}
We have the following natural isomorphisms of complexes.
\begin{itemize}
\item[(1)] $R\Hom_{\CC_A}(k,M):=\Hom_{\CC_A}(\hat{B}_{\bullet}A,M)={^*\Hom}_A(B_{\bullet}A,M)$;
\item[(2)] $R\Hom_{\CC_{\gr A}}(k,\gr M):=\Hom_{\CC_{\gr A}}(\gr B_{\bullet}A,\gr M)={^*\Hom}_{\gr A}(gr B_{\bullet}A,\gr M)$.
\end{itemize}
\end{lem}

\begin{proof}
Part (2) follows straight from the definition of morphisms in $\CC_{\gr A}$. For part (1) we have to show that $\phi \in \Hom_A(B_nA,M)$ extends to a continuous map
$\hat{B}_nA\rightarrow M$ if and only if $\phi(\Fil^i B_nA)=0$ for large $i$. Note that $\hat{B}_nA$ is the completion of $B_nA$ for the topology defined by the submodules
$$
J_i:=\ker \big(A^{\otimes(n+1)} \longrightarrow (A/\Fil^i A)^{\otimes(n+1)}\big)=(\Fil^iA \otimes_k\cdots \otimes_k A)+\cdots +(A \otimes_k \cdots \otimes_k \Fil^iA).
$$
Obviously $J_i \subset \Fil^i B_nA$, and conversely $\Fil^{i(n+1)}B_nA \subset J_i$. (If $i_0+\cdots+ i_n=i(n+1)$ we must have some $i_j\geq i$ 
and thus $\Fil^{i_0}A \otimes_k \Fil^{i_1}A \otimes_k \cdots \otimes_k \Fil^{i_n}A$ is contained in $J_i$. This proves the lemma.
\end{proof}

Note that if $A$ is Noetherian then the cohomology of $R\Hom_{\CC_A}(k,M)$ computes $\Ext_A^i(k,M)$, and similarly for $R\Hom_{\CC_{\gr A}}(k,\gr M)$
(assuming $\gr A$ is Noetherian): $k$ has a resolution by finite free $A$-modules, and any $A$-linear map 
$A^r \rightarrow A^s$ is automatically continuous. 

\subsection{The spectral sequence for $\Ext$}

Combining Lemma \ref{grhom} and \ref{starhom} we arrive at the following result.

\begin{thm}\label{rhom}
With notation as above, there is a natural isomorphism of graded complexes 
$$
\gr R\Hom_{\CC_A}(k,M)\overset{\sim}{\longrightarrow} R\Hom_{\CC_{\gr A}}(k,\gr M).
$$
\end{thm}

We extract a spectral sequence relating the $\Ext$-functors over $A$ and $\gr A$.

\begin{cor}\label{spec}
Assume $\dim_k \Ext_{\CC_{\gr A}}^n(k,\gr M)$ is finite for all $n$, and zero for $n$ sufficiently large. Then there is a convergent spectral sequence, which is functorial in $A$,
$$
E_1^{i,j}=\Ext_{\CC_{\gr A}}^{i,j}(k,\gr M) \Longrightarrow \Ext_{\CC_A}^{i+j}(k,M).
$$
(The definition of the internal grading $\Ext_{\CC_{\gr A}}^{i,j}$ is recalled in the proof -- see equation (\ref{intern}) below.)
\end{cor}

\begin{proof}
Consider the spectral sequence of the filtered complex $K={^*\Hom}_A(B_{\bullet}A,M)$, cf. \cite[p.~203]{GM03}. It consists of bigraded groups
$E_r=\oplus_{i,j}E_r^{i,j}$ starting with $E_1^{i,j}=H^{i+j}(\gr^i K)$ where $\gr K=\oplus_{i \in \Z}\gr^i K$ is the associated graded complex, which in our case equals
${^*\Hom}_{\gr A}(gr B_{\bullet}A,\gr M)$ by Theorem \ref{rhom}. Its cohomology $\Ext_{\CC_{\gr A}}^{i+j}(k,\gr M)$ is graded and $E_1^{i,j}$ sits as the $i$th graded piece;
\begin{equation}\label{intern}
E_1^{i,j}=\gr^i \Ext_{\CC_{\gr A}}^{i+j}(k,\gr M)=: \Ext_{\CC_{\gr A}}^{i,j}(k,\gr M).
\end{equation}
On the infinite sheet we have groups $E_{\infty}=\oplus_{i,j}E_{\infty}^{i,j}$ with $E_{\infty}^{i,j}=\gr^i H^{i+j}(K)$, where the filtration on $H^n(K)=\Ext_{\CC_A}^n(k,M)$ is given by the images of the maps $H^n(\Fil^i K) \rightarrow H^n(K)$. Convergence follows easily from the finiteness assumptions: $E_1^{i,j}$ is nonzero only for finitely many pairs $(i,j)$, so eventually the incoming and outgoing differentials vanish at $(i,j)$ for $r>>0$.
\end{proof}

In most situations of interest $\gr A$ will be Noetherian, and $k$ of finite projective dimension over $\gr A$. Then the finiteness conditions in \ref{spec} are automatically satisfied for all finite-dimensional $M$. 

Following \cite{PT02} Ronchetti has recently established a spectral sequence akin to the one in corollary \ref{spec} above, albeit in a more restricted setting, cf. \cite[Thm.~3]{Ron18}. He considers $\Z_p[\![U(\OO)]\!]$ filtered by powers of the augmentation ideal and relates $\Ext_{\gr \Z_p[\![U(\OO)]\!]}^*(\gr \Z_p, \gr M)$ to $\Ext_{\Z_p[\![U(\OO)]\!]}^*(\Z_p,M)$ for finitely generated $\Z_p$-modules $M$ with a continuous $B(\OO)$-action.

\subsection{An application}

Let $A' \rightarrow A$ be a filtration-preserving map of augmented $k$-algebras, both assumed to be pseudocompact. As pointed out earlier this gives rise to a morphism of filtered complexes $R\Hom_{\CC_A}(k,M)\rightarrow R\Hom_{\CC_{A'}}(k,M)$, and a fortiori a map of spectral sequences -- which we will assume converge.

\begin{cor}\label{fil}
Fix an $n$ and suppose the restriction map $\Ext_{\CC_{\gr A}}^n(k,\gr M)\rightarrow \Ext_{\CC_{\gr A'}}^n(k,\gr M)$ is zero. Then 
the map $\Ext_{\CC_{A}}^n(k,M)\rightarrow \Ext_{\CC_{A'}}^n(k,M)$ is zero on all graded pieces. That is, for all $i \in \Z$ it maps 
$$
\Fil^i \Ext_{\CC_{A}}^n(k,M)\rightarrow \Fil^{i+1} \Ext_{\CC_{A'}}^n(k,M).
$$
\end{cor}

\begin{proof}
The map between the first sheets $f_1^{i,j}: E_1^{i,j}\rightarrow E_1'^{i,j}$ comes from the restriction maps
$$
\Ext_{\CC_{\gr A}}^{i+j}(k,\gr M) \longrightarrow \Ext_{\CC_{\gr A'}}^{i+j}(k,\gr M)
$$
by taking the $i$th graded piece. Analogously the map $f^n: \Ext_{\CC_A}^n(k,M)\rightarrow \Ext_{\CC_{A'}}^n(k,M)$ is the obvious map obtained by taking the cohomology of 
the map of filtered complexes above $K \rightarrow K'$. The map $f^n$ preserves filtrations and the induced map on graded pieces ($n=i+j$)
$$
E_{\infty}^{i,j}=\gr^i \Ext_{\CC_A}^{i+j}(k,M) \overset{f^{i+j}}{\longrightarrow} \gr^i \Ext_{\CC_{A'}}^{i+j}(k,M)=E_{\infty}'^{i,j}
$$
coincides with the map between the infinite sheets $f_{\infty}^{i,j}$ which is zero since $f_1^{i,j}=0$ by assumption.
\end{proof}

Now suppose instead of a single $A'$ we have a whole system of pseudocompact filtered $k$-algebras $A^{(m)}$ along with filtration-preserving maps of augmented $k$-algebras
$$
A=:A^{(0)} \longleftarrow A^{(1)}  \longleftarrow \cdots  \longleftarrow A^{(m)}  \longleftarrow \cdots.
$$
(Typically the $A^{(m)}$ will be subalgebras of $A$.) Suppose furthermore that all the graded restriction maps vanish for a given $n$. I.e.,
\begin{equation}\label{van}
\Ext_{\CC_{\gr A}}^n(k,\gr M)\overset{0}{\longrightarrow} \Ext_{\CC_{\gr A^{(1)}}}^n(k,\gr M) \overset{0}{\longrightarrow} \cdots \overset{0}{\longrightarrow} \Ext_{\CC_{\gr A^{(m)}}}^n(k,\gr M) \overset{0}{\longrightarrow} \cdots.
\end{equation}
We arrive at one of the main results in this appendix, which we will strengthen in the next section. 

\begin{cor}\label{final}
Under the assumption (\ref{van}) for a given $n$,
$$
\varinjlim_m \Ext_{\CC_{A^{(m)}}}^n(k,M)=0.
$$
\end{cor}

\begin{proof}
It suffices to show that any $c \in \Ext_{\CC_{A}}^n(k,M)$ maps to $0$ in $\Ext_{\CC_{A^{(m)}}}^n(k,M)$ for $m$ large enough (how large may depend on $c$). Keep the notation
$K={^*\Hom}_A(B_{\bullet}A,M)$. By definition of ${^*\Hom}$ the filtration on $K$ is exhaustive; $K=\cup_{s\in \Z}\Fil^s K$. Also, by definition of the graded cohomology
$H^n(\Fil^i K)\twoheadrightarrow \Fil^i H^n(K)$, so the filtration on $H^n(K)=\Ext_{\CC_{A}}^n(k,M)$ is also exhaustive. Thus any $c$ lies in $\Fil^i\Ext_{\CC_{A}}^n(k,M)$ for some $i$.
By repeated use of Corollary \ref{fil} restriction to $A^{(m)}$ takes 
$$
\Fil^i\Ext_{\CC_{A}}^n(k,M) \longrightarrow \Fil^{i+m}\Ext_{\CC_{A^{(m)}}}^n(k,M).
$$
It suffices to observe the target is zero as long as $\Fil^{i+m}M=0$ (indeed $\Fil^{\mu}M=0$ implies $\Fil^{\mu}K=0$, and hence
$\Fil^{\mu}H^n(K)=0$). As long as we take $m \geq \mu-i$ the class $c$ restricts to zero (note that $i$ may depend on $c$).
\end{proof}

The next and last section of this appendix deals with the uniformity of $m$. Under favorable circumstances one can choose an $m$ which is independent of $c$ and thereby show that 
the restriction map $\Ext_{\CC_{A}}^n(k,M) \rightarrow \Ext_{\CC_{A^{(m)}}}^n(k,M)$ vanishes; still assuming (\ref{van}) of course. This is a key point on which our paper relies.

\subsection{Minimal resolutions and Koszul algebras}

We assume $\gr A$ is Noetherian and $k \overset{\sim}{\longrightarrow} \gr^0 A$. Then the trivial module $k$ admits a minimal projective resolution 
$$
\cdots \longrightarrow P_2 \longrightarrow P_1 \longrightarrow P_0 \longrightarrow k \longrightarrow 0.
$$
Here each $P_n\simeq \gr A\otimes_k V_n$ for some graded finite-dimensional vector space $V_n$, say $\dim_k V_n=b_n$. Choosing a basis of homogeneous elements for $V_n$ 
we can identify $P_n\simeq (\gr A)^{b_n}$ in such a way that the standard basis vectors are homogeneous elements of various degrees. Then the minimal resolution becomes
$$
\cdots \overset{T_3}\longrightarrow (\gr A)^{b_2} \overset{T_2}\longrightarrow  (\gr A)^{b_1} \overset{T_1}{\longrightarrow} (\gr A)^{b_0}=\gr A \overset{\epsilon}\longrightarrow k \longrightarrow 0,
$$
where the differentials are left-multiplication by certain matrices $T_i \in M_{b_{i-1}\times b_i}(\gr A)$. Minimality refers to the fact that the entries of all the $T_i$ in fact lie in the augmentation ideal $(\gr A)_+=\oplus_{i>0}\gr^i A$. In other words that $P_{\bullet}\otimes_{\gr A}k$ has $0$-differentials so its terms give the cohomology. Similarly for 
$\Hom_{\gr A}(P_{\bullet},k)$. Recall that $\gr A$ is said to be a {\it{Koszul}} algebra if the entries of the $T_i$ all lie in $\gr^1 A$. A convenient reference for all this is 
\cite{Kra17}. Note however that our bigrading $\Ext_{\CC_{\gr A}}^{i,j}$ differs from his internal grading; our $\Ext^{i,j}$ would be $\Ext^{i+j,i}$ in Kr\"{a}hmer's notation. 

The minimal resolution is a resolution as graded modules, so all the $T_i$ preserve the grading. Let us focus on the first $T_1$ for a second; a row vector with $b_1$ entries in 
$(\gr A)_+$. Let $e \in  (\gr A)^{b_1}$ be one of the standard basis vectors. It has some degree $|e|=d$. Since $T_1$ is of degree zero $|T_1(e)|=d$. However, $T_1(e)$ is one of the entries of $T_1$ which are in $(\gr A)_+$. We infer that $d>0$. In other words that $(\gr A)^{b_1}$ has a basis consisting of elements of degree $\geq 1$ (of degree {\it{equal}} to $1$ if $\gr A$ is Koszul). More generally we obtain the following.

\begin{lem}\label{degbas}
$P_n\simeq (\gr A)^{b_n}$ has a $\gr A$-basis consisting of homogeneous elements all of degree $\geq n$ (all of degree exactly $n$ if $\gr A$ is Koszul). 
\end{lem}

\begin{proof}
Induction on $n$. We did the case $n=1$ above. Let $n>1$ and consider $T_n: (\gr A)^{b_n}\rightarrow (\gr A)^{b_{n-1}}$. Let $e \in (\gr A)^{b_n}$ be the $j$th standard basis vector,
of some degree $|e|=d$. Then $T_n(e)$ is the $j$th column of the matrix $T_n=(a_{ij})$. That is
$$
T_n(e)=a_{1j}e_1+a_{2j}e_2+\cdots+a_{b_{n-1},j}e_{b_{n-1}}
$$
where all the $e_i$'s have degree $\geq n-1$ (with equality in the Koszul case) by induction and $a_{ij}\in (\gr A)_+$. This shows $T_n(e)$ is a sum of terms of degree $\geq n$. 
Since we know it is homogeneous, $|e|=|T_n(e)|\geq n$. In the Koszul case all terms of $T_n(e)$ have degree exactly $n$ by induction since all $a_{ij}\in \gr^1 A$.
\end{proof}

Now let $M$ be a filtered $A$-module as above, with associated graded $\gr A$-module $\gr M$. Since $P_{\bullet}$ and $\gr B_{\bullet} A$ are homotopy equivalent as graded resolutions of $k$, the grading of $\Ext_{\CC_{\gr A}}^n(k,\gr M)$ is also determined by that of $P_{\bullet}$. Thus by its minimality we have
$$
E_1^{i,j}=\Ext_{\CC_{\gr A}}^{i,j}(k,\gr M)=H^{i+j}(\Hom_{\gr A}^i(P_{\bullet},\gr M))=\Hom_{\gr A}^i(P_{i+j},\gr M).
$$
Let $\mu$ be the smallest integer for which $\Fil^{\mu}M=0$, and let $\nu$ be the largest integer for which $\Fil^{\nu}M=M$. In particular
$\gr^sM=0$ when $s$ lies outside the interval $[\nu,\mu]$. We define the {\it{amplitude}} of $M$ as the length of the filtration, i.e. $\text{amp}(M)=\mu-\nu$.

\begin{thm}\label{koz}
Suppose $\gr A$ is a \underline{Koszul} algebra and the vanishing hypothesis (\ref{van}) is satisfied for a given $n\geq 1$. Then the restriction map 
$$
\Ext_{\CC_A}^{n}(k,M) \longrightarrow \Ext_{\CC_{A^{(m)}}}^{n}(k,M)
$$
vanishes for $m>\text{amp}(M)+n$.
\end{thm}

\begin{proof}
First we observe that on the initial page of the spectral sequence
$$
\text{$E_1^{i,j}=0$ in the region $2i+j \geq \mu$, \underline{and} in $2i+j<\nu$ since $\gr A$ is Koszul.}
$$
Indeed, by Lemma \ref{degbas} we know $P_{i+j}$ is generated by elements $e$ of degree $\geq i+j$. Any degree $i$ map $\phi: P_{i+j}\rightarrow \gr M$ takes $e$ to an element
$\phi(e)$ of degree $\geq 2i+j$. Assuming $2i+j \geq \mu$ we must have $\phi=0$. When $\gr A$ is Koszul $|\phi(e)|=2i+j$ and therefore $\phi=0$ if $2i+j<\nu$. 

Consequently, on the infinite page $E_{\infty}^{i,j}=\gr^i \Ext_{\CC_A}^{i+j}(k,M)$ also vanishes when $2i+j \geq \mu$ or $2i+j<\nu$.
In other words, for a fixed $n$, the filtration $\Fil^i \Ext_{\CC_A}^{n}(k,M)$ becomes stationary and therefore zero for $i \geq \mu-n$, and it equals 
$\Ext_{\CC_A}^{n}(k,M)$ for $i<\nu-n$ since we know the filtration is exhaustive as observed earlier. This allows us to strengthen the conclusion of Corollary \ref{final} when $\gr A$ is Koszul. Its proof shows that the restriction map in question is zero as long as $m>\mu-\nu+n$.
\end{proof}

The uniformity of how deep we have to go to get vanishing of the restriction map plays a critical role in this paper. 






\subsection*{Acknowledgments} I would like to thank Julien Hauseux, Karol Koziol, and Jake Postema for numerous helpful conversations and correspondence related to this work, and Peter Schneider for hosting me during an inspiring visit to M\"{u}nster in July 2015. Also many thanks to Fred Diamond, Michael Harris, and Marie-France Vigneras for their feedback on the initial announcement of the results of this paper. After circulating a first draft I received very helpful feedback from Noriyuki Abe, Jan Kohlhaase, Karol Koziol and Niccolo' Ronchetti which led to significant improvements. I am grateful for their careful reading of the manuscript. Last but not least a big thank you to the anonymous referee for meticulously combing through the manuscript and returning corrections and thoughtful suggestions.



\bigskip

\noindent {\it{E-mail address}}: {\texttt{csorensen@ucsd.edu}}

\noindent {\sc{Department of Mathematics, UCSD, La Jolla, CA, USA.}}

\end{document}